\documentclass[secthm,seceqn,amsthm,ussrhead,reqno, 12pt]{amsart}
\usepackage[utf8]{inputenc}
\usepackage[english]{babel}
\usepackage[symbol]{footmisc}
\usepackage{amssymb,amsmath,amsthm,amsfonts,xcolor,enumerate,hyperref,comment,longtable,cleveref}

\usepackage{times}
\usepackage{cite}
\usepackage{pdflscape}
\usepackage{ulem}
\usepackage[mathcal]{euscript}
\usepackage{tikz}
\usepackage{hyperref}
\usepackage{cancel}
\usepackage{stmaryrd}

\usetikzlibrary{arrows}

\newtheorem{theorem}{Theorem}
\newtheorem{lemma}[theorem]{Lemma}

\newtheorem{definition}[theorem]{Definition}

\usepackage{stmaryrd}
\usepackage{xcolor}

\begin{document}

\noindent{\Large
Transposed Poisson structures on quasi-filiform Lie algebras of \\
maximum length}\footnote{
This work was supported by UIDB/MAT/00212/2020,  UIDP/MAT/00212/2020, 2022.02474.PTDC and by grant F-FA-2021-423, Ministry of Higher Education, Science and Innovations of the Republic of Uzbekistan.
}

 \bigskip

\begin{center}

 {\bf
Kobiljon Abdurasulov\footnote{CMA-UBI, Universidade da Beira Interior, Covilh\~{a}, Portugal; \ Institute of Mathematics Academy of
Sciences of Uzbekistan, Tashkent, Uzbekistan; \ abdurasulov0505@mail.ru},
Fatanah Deraman\footnote{Universiti Malaysia Perlis, Perlis, Malaysia; \ fatanah@unimap.edu.my},
Azamat Saydaliyev \footnote{Institute of Mathematics Academy of
Sciences of Uzbekistan, Tashkent, Uzbekistan; National University of Uzbekistan, Tashkent, Uzbekistan; \ azamatsaydaliyev6@gmail.com}
  \&
Siti Hasana Sapar\footnote{Universiti Putra Malaysia, Selangor, Malaysia; \ sitihas@upm.edu.my}

}

\end{center}

\

\medskip

\noindent{\bf Abstract}:
{\it This article will discussing on $\frac{1}{2}$-derivations of  quasi-filiform Lie algebras of maximum length. The non-trivial transposed Poisson algebras with the quasi-filiform Lie algebras of maximum length are constructed by using $\frac{1}{2}$-derivations of Lie algebras. We have established commutative associative multiplication to construct a transposed Poisson algebra with an associated given Lie algebra.}

\medskip

\bigskip



 \bigskip

\section*{Introduction}

Bai et.al.~\cite{Bai} introduced a dual notion of the Poisson algebra, which is called a \textit{transposed Poisson algebra}, by changing the roles of the two multiplications in the Leibniz rule.
Poisson algebra is introduced to commutative associative algebras by their derivation. Similarly, on a Lie algebra, the concept of a transposed Poisson algebra is defined by the $\frac{1}{2}$-derivation. This way of defining not only shares some properties of a Poisson algebra, such as the closedness under tensor products and the Koszul self-duality as an operation, but also admits a rich class of identities \cite{kms,Bai,fer23,bfk22,lb23,bl23,conj}.

The description of all Poisson algebras with a fixed Lie or associative part~\cite{YYZ07,jawo,kk21} is one of the natural tasksm in the theory.
In this paper, we present transposed Poisson algebras that demonstrate quasi-filiform Lie algebras of maximum length.
Abdurasulov et.al.~\cite{Kobiljon}, obtained the description of all transposed Poisson algebra structures on solvable Lie algebras with filiform nilradical.

 It is worth noting that any unital transposed Poisson algebra is a particular case of a ``contact bracket'' algebra	and a quasi-Poisson algebra \cite{bfk22}. Every transposed Poisson algebra is a commutative  Gelfand-Dorfman algebra \cite{kms}, and is also an algebra of Jordan brackets \cite{fer23}. Ferreira et. al. \cite{FKL} established a relation between transposed Poisson algebras and $\frac{1}{2}$-derivations of Lie algebras. These ideas played an active role in describing all transposed Poisson structures on  Witt and Virasoro algebras in \cite{FKL}.  Twisted Heisenberg-Virasoro,   Schr\"odinger-Virasoro and extended Schr\"odinger-Virasoro algebras were studied in \cite{yh21}. Meanwhile, Schr\"odinger algebra in $(n+1)$-dimensional space-time was discussed by \cite{ytk}. Then, in \cite{kk23} Witt type Lie algebras were discussed and \cite{kkg23} studied generalized Witt algebras. In the paper by \cite{kk22,kkg23}, they were studied on Block Lie algebras and \cite{KK7} was studied on Lie algebra of upper triangular matrices and  a transposed Poisson structures on Lie algebra of upper triangular matrices, while \cite{kkinc} was discussed on Lie incidence algebras. It was proved in \cite{klv22} that any complex finite-dimensional solvable Lie algebra admit a non-trivial transposed Poisson structure. In \cite{bfk23} authors gave the algebraic and geometric classification of three-dimensional transposed Poisson algebras. The list of actual open questions on transposed Poisson algebras can be seen in \cite{bfk22}. Recently, all transposed Poisson algebra structures on oscillator Lie algebras, i.e., on one-dimensional solvable extensions of  $(2n+1)$-dimensional Heisenberg algebra; on solvable Lie algebras with naturally graded filiform nilpotent radical; on $(n+1)$-dimensional solvable extensions of the $(2n + 1)$-dimensional Heisenberg algebra; and on $n$-dimensional solvable extensions of   $n$-dimensional algebra with the trivial multiplication were described \cite{KKh}. This paper also gave an example of a finite-dimensional Lie algebra with non-trivial $\frac{1}{2}$-derivations without non-trivial transposed Poisson algebra structures.	
Also, similar studies can be seen in \cite[Section 7.3]{k23}, and in the references given in that paper.
In \cite{bfk23}, they were obtained the algebraic and geometric classification of all complex $3$-dimensional transposed Poisson algebras, and \cite{ch24} was studied on the algebraic classification of all complex $3$-dimensional transposed Poisson $3$-Lie algebras.

The purpose of this article is to find all transposed Poisson algebras that demonstrate quasi-filiform Lie algebras of maximum length. In order to achieve our goal, we have organized this paper as follows: in Section 2, we described $\frac{1}{2}$-derivations of quasi-filiform Lie algebras of maximum length and Section 3, will describe all non-trivial transposed Poisson algebras with quasi-filiform Lie algebras of maximum length. Next, by using descriptions of $\frac{1}{2}$-derivations of Lie algebras, we established commutative associative multiplication to construct a transposed Poisson algebra associated with given Lie algebra.

\bigskip

\section{PRELIMINARIES}

This section will discuss the concepts and known results on algebras over the field $\mathbb{C}$ unless otherwise stated.

We start with the definition of a Poisson algebra.

\begin{definition}
Let $\mathfrak{L}$ be a vector space equipped with two bilinear operations
$$
\cdot,\; [-,-] :\mathfrak{L}\otimes \mathfrak{L}\to \mathfrak{L}.$$
The triple $(\mathfrak{L},\cdot,[-,-])$ is called a
\textbf{Poisson algebra} if $(\mathfrak{L},\cdot)$ is a commutative associative algebra and
$(\mathfrak{L},[-,-])$ is a Lie algebra which satisfies the compatibility condition
\begin{equation}\label{eq:LR}
[x,y\cdot z]=[x,y]\cdot z+y\cdot [x,z].
\end{equation}
\end{definition}

Eq.~(\ref{eq:LR}) is called {\bf Leibniz
rule} since the adjoint operators of Lie algebra are
derivations of the commutative associative algebra.

\begin{definition}\label{TPA}
Let $\mathfrak{L}$ be a vector space equipped with two bilinear operations
$$
\cdot,\; [-,-] :\mathfrak{L}\otimes \mathfrak{L}\to \mathfrak{L}.$$
The triple $(\mathfrak{L},\cdot,[-,-])$ is called a \textbf{transposed Poisson algebra} if $(\mathfrak{L},\cdot)$ is a commutative associative algebra and $(\mathfrak{L},[-,-])$ is a Lie algebra which satisfies the following compatibility condition
\begin{equation}
2z\cdot [x,y]=[z\cdot x,y]+[x,z\cdot y].\label{eq:dualp}
\end{equation}
\end{definition}

Eq.~\eqref{eq:dualp} is called {\bf transposed
Leibniz rule} because the roles of two binary operations in the Leibniz rule in a Poisson algebra are switched. Further, the resulting operation is rescaled by introducing a factor 2 on the left-hand side.

Transposed Poisson algebras were initially introduced by Bai et. al. \cite{Bai}. A transposed Poisson structure $\cdot$ on $\mathfrak{L}$ is called \textbf{trivial}, if $x\cdot y=0$ for all $x,y\in\mathfrak{L}$.
%
%

\begin{definition}\label{halfderiv}
  Let $(\mathfrak{L}, [-,-])$ be an algebra with a multiplication $[-,-],$ and $\varphi$ be
a bilinear map. Then $\varphi$ is a $\frac12$-\textbf{derivation} if it satisfies:
\begin{equation}\label{halfderiv1}
\varphi([x, y]) = \frac12([\varphi(x), y] + [x, \varphi(y)]).
\end{equation}
\end{definition}

Note that $\frac{1}{2}$-derivations are a particular case of $\delta$-derivations introduced by Filippov in \cite{fil1}
(see, \cite{k07,k10} and references therein).  It is easy to see from Definition \ref{halfderiv} that $[\mathfrak{L},\mathfrak{L}]$ and $Ann(\mathfrak{L})$ are invariant under any $\frac 12$-derivation of $\mathfrak{L}$. Definitions \ref{TPA} and \ref{halfderiv} immediately imply the following lemma.

\begin{lemma}\label{lemma1}
  Let $(\mathfrak{L}, [-,-])$ be a Lie algebra and $\cdot$ a new binary (bilinear) operation on $\mathfrak{L}$. Then
$(\mathfrak{L}, \cdot, [-,-])$ is a transposed Poisson algebra if and only if $\cdot$ is commutative and associative and for
every $z \in \mathfrak{L}$ the multiplication by $z$ in $(\mathfrak{L}, \cdot)$ is a $\frac12$-derivation of $(\mathfrak{L}, [-,-])$.
\end{lemma}

%
%

Lower central series for a given Lie algebra $\mathfrak{L}$ is  defined as follows:
\[\mathfrak{L}^1=\mathfrak{L}, \ \mathfrak{L}^{k+1}=[\mathfrak{L}^k,\mathfrak{L}],  \ k \geq 1.\]

\begin{definition}
 A Lie algebra $\mathfrak{L}$ is said to be \textbf{nilpotent}, if there exists $k\in\mathbb N$ such that $\mathfrak{L}^{k}=\{0\}$. The minimal number $k$ with such property is said to be the \textbf{index of nilpotency} of the algebra $\mathfrak{L}$.
\end{definition}

\begin{definition}
An $n$-dimensional Lie algebra is called \textbf{quasi-filiform} if its index of nilpotency is equal to $n-1$.
\end{definition}

Now, let us define a maximum length for the nilpotent Lie algebras.

A Lie algebra $\mathfrak{L}$ is $\mathbb{Z}$-graded, if $\mathfrak{L} =\bigoplus_{i\in \mathbb{Z}}\mathbf{V}_i$, where $[V_i, V_j]\subseteq V_{i+j}$ for any $i, j \in \mathbb{Z}$ with a finite number of non-null spaces $V_i$. We say that a nilpotent Lie algebra $\mathfrak{L}$ admits \textbf{the connected gradation} $\mathfrak{L} = V_{k_1}\oplus\dots\oplus V_{k_t}$, if $V_{k_i}\neq\{0\}$ for any $i \  (1\leq i\leq t)$.

\begin{definition}
The number $l(\oplus \mathfrak{L})=l(V_{k_1}\oplus\dots\oplus V_{k_t})=k_t-k_1 + 1$ is called \textbf{the
length of gradation}. A gradation is called \textbf{of maximum length}, if $l(\oplus \mathfrak{L}) = \dim(\mathfrak{L})$.

\end{definition}

We denote $l(\mathfrak{L}) = \max\{l(\oplus L)\ \text{such that}\  L = V_{k_1}\oplus\dots\oplus V_{k_t}\  \text{is a connected gradation}\}$ of \textbf{the length of an algebra} $\mathfrak{L}$.

\begin{definition}
 A Lie algebra $\mathfrak{L}$ is called of \textbf{maximum length} if $l(\mathfrak{L}) = \dim(\mathfrak{L})$.
\end{definition}

In the following theorem we present the classification of quasi-filiform Lie algebras of maximum length given in \cite{Gomez}.

\begin{theorem} \label{thm26}
 Let $L$ be an $n$-dimensional quasi-filiform Lie algebra
of maximum length. Then the algebra $L$ is isomorphic to one of the following pairwise non-isomorphic algebras:

\[g^1_{(n,1)}:\begin{cases}
[e_1,e_i]=e_{i+1},& 2\leq i\leq n-2,\\[1mm]
[e_i,e_{n-i}]=(-1)^ie_{n},& 2\ \leq i\leq  \frac{n-1}{2}, \ n\geq5 \ \  and \ \ n \  is \  odd;\\[1mm]
\end{cases}\]

\[g^2_{(n,1)}:\begin{cases}
[e_1,e_i]=e_{i+1},& 2\leq i\leq n-2,\\[1mm]
[e_i,e_n]=e_{i+2},& 2\leq i\leq n-3, \ n\geq5; \\[1mm]
\end{cases} \quad g^3_{(n,1)}:\begin{cases}
[e_1,e_i]=e_{i+1},& 2\leq i\leq {n-2},\\[1mm]
[e_i,e_{n}]=e_{i+2},& 2\leq i\leq {n-3}\\[1mm]
[e_2,e_i]=e_{i+3}, & 3\leq i \leq {n-4}, \ n\geq7;
\end{cases}\]
\[g^1_{7}:\begin{cases}
[e_1,e_i]=e_{i+1},& 2\leq i\leq 5,\\[1mm]
[e_2,e_i]=e_{i+2},& 3\leq i\leq 4,\\[1mm]
[e_i,e_{7-i}]=(-1)^ie_{7},& 2\leq i\leq 3;\\[1mm]
\end{cases} \quad g^2_{9}:\begin{cases}
[e_1,e_i]=e_{i+1},& 2\leq i\leq 7,\\[1mm]
[e_2,e_i]=e_{i+2},& 3\leq i\leq 4 ,\\[1mm]
[e_2,e_5]=3e_7,\\[1mm]
[e_2,e_6]=5e_8,\\[1mm]
[e_3,e_i]=-2e_{i+3},&4\leq i\leq 5,\\[1mm]
[e_i,e_{9-i}]=(-1)^ie_{9},&2\leq i \leq 4;
\end{cases}\]
\[g^3_{11}:\begin{cases}
[e_1,e_i]=e_{i+1},& 2\leq i\leq 9,\\[1mm]
[e_2,e_i]=e_{i+2},& 3\leq i\leq 4 ,\\[1mm]
[e_2,e_i]=-e_{i+2},& 6\leq i\leq 7\\[1mm]
[e_3,e_7]=-e_{10},\\[1mm]
[e_3,e_i]=e_{i+3},&4\leq i\leq 5,\\[1mm]
[e_4,e_i]=e_{i+4},&5\leq i\leq 6,\\[1mm]
[e_i,e_{11-i}]=(-1)^ie_{11},&2\leq i \leq 5,
\end{cases}\]
where $\{e_1,e_2,\dots,e_{n}\}$ is a basis of the algebra.
\end{theorem}

\section{$\frac12$-derivation of  quasi-filiform Lie algebras of maximum length}

In this section, we will derive and calculate $\frac12$-derivation of  quasi-filiform Lie algebras of maximum length.
We start with the following theorem.

\begin{theorem}\label{halfderiv1} Any $\frac 1 2$-derivations of the algebra $g^1_{(n,1)}$
has the form

\begin{itemize}
  \item for $n=5:$
  $$\begin{cases}
\varphi(e_1)=\alpha_1e_1+\alpha_2e_2+\alpha_3e_3+\alpha_4e_4+\alpha_5e_5, \\
\varphi(e_2)=\beta_1e_1+\beta_2e_2+\beta_3e_3+\beta_4e_4+\beta_5e_5,\\
\varphi(e_3)=\frac{1}{2}(\alpha_1+\beta_2)e_3+\frac{1}{2}\beta_3e_4-\frac{1}{2}\alpha_3e_5, \\
\varphi(e_4)=\frac{1}{4}(3\alpha_1+\beta_2)e_4+\frac{1}{2}\alpha_2e_5, \\
\varphi(e_5)=\frac{1}{2}\beta_1e_4+\frac{1}{4}(\alpha_1+3\beta_2)e_5, \\
  \end{cases}$$
  \item for $n\geq7:$
$$\begin{cases}
\varphi(e_1)=\sum\limits_{i=1}^{n}\alpha_ie_i, \\ \varphi(e_2)=\alpha_1e_2+\beta_{n-2}e_{n-2}+\beta_{n-1}e_{n-1}+\beta_ne_n, \\  \varphi(e_3)=\alpha_1e_3+\frac{1}{2}\beta_{n-2}e_{n-1}-\frac{1}{2}\alpha_{n-2}e_{n}, \\
\varphi(e_i)=\alpha_1e_i+\frac{(-1)^i}{2}\alpha_{n-i+1}e_n, \  4\leq i\leq n-1,\\
\varphi(e_n)=\alpha_1e_n.\end{cases}$$\end{itemize}
\end{theorem}

\begin{proof}
It is easy to see that $g^1_{(n,1)}$ has $2$ generators. We use these generators to calculate of $\frac12$-derivation.
$$\varphi(e_1)=\sum\limits_{i=1}^{n}\alpha_ie_i,\quad \varphi(e_2)=\sum\limits_{i=1}^{n}\beta_ie_i.$$

Now consider the condition of $\frac 1 2$-derivation for the elements $e_1$ and $e_2:$
\begin{center}
$\varphi(e_3)=\varphi([e_1,e_2])=\frac{1}{2}([\varphi(e_1),e_2]+[e_1,\varphi(e_2)])$
\end{center}
\begin{center}
$=\frac{1}{2}([\sum\limits_{i=1}^{n}\alpha_ie_i,e_2]+[e_1,\sum\limits_{i=1}^{n}\beta_ie_i])
=\frac{1}{2}((\alpha_1+\beta_2)e_3+\sum\limits_{i=4}^{n-1}\beta_{i-1}e_i-\alpha_{n-2}e_n).$
\end{center}

We prove the following equality for $3\leq i\leq n-1$ by induction:
\begin{center}
$\varphi(e_i)=\frac{1}{2^{i-2}}\Big((2^{i-2}-1)\alpha_{1}+\beta_2\Big)e_{i}+\frac{1}{2^{i-2}}\sum\limits_{t=i+1}^{n-1}\beta_{t-i+2}e_t+\frac{(-1)^{i}}{2}\alpha_{n-i+1}e_n.$
\end{center}
If $i=3$, the relationship holds according to the above equality. Now, we prove that it is true for $i$ and $i+1$.
By considering the condition of $\frac 1 2$-derivation for the elements $e_1, e_i$ we have
\begin{center}
$\varphi(e_{i+1})=\varphi([e_1,e_i])=\frac{1}{2}([\varphi(e_1),e_i]+[e_1,\varphi(e_i)])$
\end{center}
\begin{center}
$=\frac{1}{2}\Big([\sum\limits_{k=1}^{n}\alpha_ke_k,e_i]+[e_1,\frac{1}{2^{i-2}}\Big((2^{i-2}-1)\alpha_{1}+\beta_2\Big)e_{i}+\frac{1}{2^{i-2}}\sum\limits_{t=i+1}^{n-1}\beta_{t-i+2}e_t+\frac{(-1)^{i}}{2}\alpha_{n-i+1}e_n]\Big)$
\end{center}
\begin{center}
$=\frac{1}{2}\Big(\alpha_{1}e_{i+1}+(-1)^{i+1}\alpha_{n-i}e_n+\frac{1}{2^{i-2}}\Big((2^{i-2}-1)\alpha_{1}+\beta_2\Big)e_{i+1}+\frac{1}{2^{i-2}}\sum\limits_{t=i+1}^{n-2}\beta_{t-i+2}e_{t+1}\Big)$
\end{center}
\begin{center}
$=\frac{1}{2^{i-1}}\Big((2^{i-1}-1)\alpha_{1}+\beta_2\Big)e_{i+1}+\frac{1}{2^{i-1}}\sum\limits_{t=i+2}^{n-1}\beta_{t-i+1}e_t+\frac{(-1)^{i+1}}{2}\alpha_{n-i}e_n.$
\end{center}

Now, consider the condition of $\frac 1 2$-derivation for the elements $e_2, e_{n-2}:$
\begin{center}
$\varphi(e_{n})=\varphi([e_2,e_{n-2}])=\frac{1}{2}([\varphi(e_2),e_{n-2}]+[e_2,\varphi(e_{n-2})])$
\end{center}
\begin{center}
$=\frac{1}{2}\Big([\sum\limits_{i=1}^{n}\beta_ie_i,e_{n-2}]+[e_2,\frac{1}{2^{n-4}}\Big((2^{n-4}-1)\alpha_{1}+\beta_2\Big)e_{n-2}+\frac{1}{2^{n-4}}\beta_{3}e_{n-1}+\frac{(-1)^{n-2}}{2}\alpha_{3}e_n]\Big)
$
\end{center}
\begin{center}
$=\frac{1}{2}\Big(\beta_{1}e_{n-1}+\beta_{2}e_n+\frac{1}{2^{n-4}}\Big((2^{n-4}-1)\alpha_{1}+\beta_2\Big)e_{n}\Big)=\frac{1}{2}\beta_{1}e_{n-1}+\frac{1}{2^{n-3}}\Big((2^{n-4}-1)\alpha_{1}+(2^{n-4}+1)\beta_2\Big)e_{n}.$
\end{center}
Thus, we obtain that
\begin{equation}\label{fien1}
\varphi(e_{n})=\frac{1}{2}\beta_{1}e_{n-1}+\frac{1}{2^{n-3}}\Big((2^{n-4}-1)\alpha_{1}+(2^{n-4}+1)\beta_2\Big)e_{n}.
\end{equation}
Now, we consider the condition of $\frac 1 2$-derivation for the elements $e_2$ and $e_i$, $3\leq i\leq n-3:$
\begin{center}
$0=\varphi([e_2,e_i])=\frac{1}{2}([\varphi(e_2),e_i]+[e_2,\varphi(e_i)])$
\end{center}
\begin{center}
$=\frac{1}{2}\Big([\sum\limits_{i=1}^{n}\beta_ie_i,e_i]+[e_2,\frac{1}{2^{i-2}}\Big((2^{i-2}-1)\alpha_{1}+\beta_2\Big)e_{i}+\frac{1}{2^{i-2}}\sum\limits_{t=i+1}^{n-1}\beta_{t-i+2}e_t+\frac{(-1)^{i}}{2}\alpha_{n-i+1}e_n]\Big)$
\end{center}
\begin{center}
$=\frac{1}{2}\Big([\beta_{1}e_{i+1}-(-1)^i\beta_{n-i}e_n+\frac{1}{2^{i-2}}\beta_{n-i}e_{n}\Big).$
\end{center}
By the comparison of coefficients at the basis elements we obtain that
$$\beta_1=\beta_i=0, \ 3\leq i\leq n-3.$$

By considering the condition of $\frac 1 2$-derivation for the elements $e_3$ and $e_{n-3}$ we have
\begin{center}
$\varphi([e_3,e_{n-3}])=\frac{1}{2}([\varphi(e_3),e_{n-3}]+[e_3,\varphi(e_{n-3})])$
\end{center}
\begin{center}
$=\frac{1}{2}\Big([\frac{1}{2}((\alpha_1+\beta_2)e_3+\beta_{n-2}e_{n-1}-\alpha_{n-2}e_n),e_{n-3}]+
[e_3,\frac{1}{2^{n-5}}\Big((2^{n-5}-1)\alpha_{1}+\beta_2\Big)e_{n-3}+\frac{1}{2}\alpha_{4}e_n]\Big)$
\end{center}
\begin{center}
$=\frac{1}{2}\Big(-\frac{1}{2}((\alpha_1+\beta_{2})e_{n}-\frac{1}{2^{n-5}}\Big((2^{n-5}-1)\alpha_{1}+\beta_2\Big)e_{n}\Big).$
\end{center}
On the other hand,
\begin{equation}\label{fien2}
\varphi([e_3,e_{n-3}])=-\varphi(e_n)=-\Big(\frac{1}{2^2}((\alpha_1+\beta_{2})+\frac{1}{2^{n-4}}\Big((2^{n-5}-1)\alpha_{1}+\beta_2\Big)\Big)e_{n}.
\end{equation}
By comparing the coefficients at the basis elements in expressions (\ref{fien1}) and (\ref{fien2}) we obtain that
$\alpha_{1}=\beta_2.$

This completes the proof of the theorem. \end{proof}

Now we study the $\frac 1 2$-derivation of the algebra $g^2_{(n,1)}.$

\begin{theorem}\label{halfderiv2} Any $\frac 1 2$-derivations of the algebra $g^2_{(n,1)}$ has the form
\begin{itemize}
  \item for $n=5:$
  $$\begin{cases}
\varphi(e_1)=\alpha_1e_1+\alpha_2e_2+\alpha_3e_3+\alpha_4e_4+\alpha_5e_5, \\
\varphi(e_2)=\beta_2e_2+\beta_3e_3+\beta_4e_4+\beta_5e_5,\\
\varphi(e_3)=\frac{1}{2}(\alpha_1+\beta_2)e_3+\frac{1}{2}(\beta_3-\alpha_5)e_4, \\
\varphi(e_4)=\frac{1}{4}(3\alpha_1+\beta_2)e_4, \\
\varphi(e_5)=-\alpha_2e_3+\gamma_4e_4+\frac{1}{2}(3\alpha_1-\beta_2)e_5, \\
  \end{cases}$$
  \item for $n=6:$
  $$\begin{cases}
\varphi(e_1)=\alpha_1e_1+\alpha_2e_2+\alpha_3e_3+\alpha_4e_4+\alpha_5e_5+\alpha_6e_6, \\
\varphi(e_2)=\alpha_1e_2-3\alpha_6e_3+\beta_4e_4+\beta_5e_5,\\
\varphi(e_3)=\alpha_1e_3-2\alpha_6e_4+\frac{1}{2}\beta_4e_5, \\
\varphi(e_4)=\alpha_1e_4-\frac{3}{2}\alpha_6e_5, \\
\varphi(e_5)=\alpha_1e_5, \\
\varphi(e_6)=-\alpha_2e_3-\alpha_3e_4+\gamma_5e_5+\alpha_1e_6, \\
  \end{cases}$$
  \item for $n\geq7:$
$$\begin{cases}
\varphi(e_1)=\sum\limits_{i=1}^{n-1}\alpha_ie_i, \\
\varphi(e_2)=\alpha_1e_2+\beta_{n-2}e_{n-2}+\beta_{n-1}e_{n-1}, \\
\varphi(e_3)=\alpha_1e_3+\frac{1}{2}\beta_{n-2}e_{n-1}, \\
\varphi(e_i)=\alpha_1e_i, \  4\leq i\leq n-1,\\
\varphi(e_n)=-\sum\limits_{i=3}^{n-2}\alpha_{i-1}e_i+\gamma_{n-1}e_{n-1}+\alpha_1e_n.\end{cases}$$\end{itemize}
\end{theorem}

\begin{proof} From table multiplication of the algebra  $g^2_{(n,1)}$ we conclude that $e_1, e_2$ and $e_{n}$ are the generator basis elements of
the algebra. We use these generators to calculate of $\frac12$-derivation:
\begin{center}
$\varphi(e_1)=\sum\limits_{i=1}^{n}\alpha_ie_i,\ \varphi(e_2)=\sum\limits_{i=1}^{n}\beta_ie_i,\ \varphi(x)=\sum\limits_{i=1}^{n}\gamma_ie_i.$
\end{center}
Now consider the condition of $\frac 1 2$-derivation for the elements $e_1$ and $e_2:$
\begin{center}
$\varphi(e_3)=\varphi([e_1,e_2])=\frac{1}{2}([\varphi(e_1),e_2]+[e_1,\varphi(e_2)])$
\end{center}
\begin{center}
$=\frac{1}{2}([\sum\limits_{i=1}^{n}\alpha_ie_i,e_2]+[e_1,\sum\limits_{i=1}^{n}\beta_ie_i])
=\frac{1}{2}\Big((\alpha_1+\beta_2)e_3+(\beta_3-\alpha_n)e_4+\sum\limits_{i=5}^{n-1}\beta_{i-1}e_i\Big).$
\end{center}

We prove the following equality for $3\leq i\leq n-2$ by induction:
\begin{center}
$\varphi(e_i)=\frac{1}{2^{i-2}}\Big((2^{i-2}-1)\alpha_{1}+\beta_2\Big)e_{i}+\frac{1}{2^{i-2}}\Big(\beta_{3}-(2^{i-2}-1)\alpha_n\Big)e_{i+1}+\frac{1}{2^{i-2}}\sum\limits_{t=i+2}^{n-1}\beta_{t-i+2}e_t$
\end{center}
and
\begin{center}
$\varphi(e_{n-1})=\frac{1}{2^{n-3}}\Big((2^{n-3}-1)\alpha_{1}+\beta_2\Big)e_{n-1}.$
\end{center}

If $i=3$, the relationship is fulfilled according to the above equality. Now, we prove that it is true for $i$ and $i+1$.
By considering the condition of $\frac 1 2$-derivation for the elements $e_1, e_i$ , we have the following:
\begin{center}
$\varphi(e_{i+1})=\varphi([e_1,e_i])=\frac{1}{2}([\varphi(e_1),e_i]+[e_1,\varphi(e_i)])$
\end{center}
\begin{center}
$=\frac{1}{2}\Big([\sum\limits_{k=1}^{n}\alpha_ke_k,e_i]+[e_1,\frac{1}{2^{i-2}}\Big((2^{i-2}-1)\alpha_{1}+\beta_2\Big)e_{i}+\frac{1}{2^{i-2}}\Big(\beta_{3}-(2^{i-2}-1)\alpha_n\Big)e_{i+1}+\frac{1}{2^{i-2}}\sum\limits_{t=i+2}^{n-1}\beta_{t-i+2}e_t\Big)$
\end{center}
\begin{center}
$=\frac{1}{2}\Big(\alpha_{1}e_{i+1}-\alpha_{n}e_{i+2}+\frac{1}{2^{i-2}}\Big((2^{i-2}-1)\alpha_{1}+\beta_2\Big)e_{i+1}+\frac{1}{2^{i-2}}\Big(\beta_{3}-(2^{i-2}-1)\alpha_n\Big)e_{i+2}$ \
$+\frac{1}{2^{i-2}}\sum\limits_{t=i+2}^{n-2}\beta_{t-i+2}e_{t+1}\Big)$
\end{center}
\begin{center}
$=\frac{1}{2^{i-1}}\Big((2^{i-1}-1)\alpha_{1}+\beta_2\Big)e_{i+1}+\frac{1}{2^{i-1}}\Big(\beta_3-(2^{i-1}-1)\alpha_{n}\Big)e_{i+2}+\frac{1}{2^{i-1}}\sum\limits_{t=i+3}^{n-1}\beta_{t-i+1}e_t.$
\end{center}

From $$0=\varphi([e_1,e_n])=\frac{1}{2}([\varphi(e_1),e_n]+[e_1,\varphi(e_n)])=\frac{1}{2}\Big([\sum\limits_{k=1}^{n}\alpha_ke_k,e_n]+[e_1,\sum\limits_{k=1}^{n}\gamma_ke_k]\Big)$$
$$=\frac{1}{2}\Big(\sum\limits_{k=2}^{n-3}\alpha_ke_{k+2}+\sum\limits_{k=2}^{n-2}\gamma_ke_{k+1}\Big)=\frac{1}{2}\Big(\sum\limits_{k=4}^{n-1}\alpha_{k-2}e_{k}+\sum\limits_{k=3}^{n-1}\gamma_{k-1}e_{k}\Big),$$ we obtain,
\begin{center}
$\gamma_2=0,\ \gamma_i=-\alpha_{i-1},\ 3\leq i\leq n-2.$
\end{center}

Now consider the condition of $\frac 1 2$-derivation for the elements $e_2, e_{n}:$
\begin{center}
$\varphi([e_2,e_{n}])=\frac{1}{2}([\varphi(e_2),e_{n}]+[e_2,\varphi(e_{n})])=\frac{1}{2}\Big([\sum\limits_{i=1}^{n}\beta_{i}e_i,e_{n}]+
[e_2,\sum\limits_{i=1}^{n}\gamma_{i}e_i]\Big)$
\end{center}
\begin{center}
$=\frac{1}{2}\Big(\sum\limits_{i=2}^{n-3}\beta_{i}e_{i+2}-\gamma_{1}e_3+\gamma_{n}e_4\Big)=\frac{1}{2}\Big(-\gamma_{1}e_3+(\beta_{2}+\gamma_{n})e_4+\sum\limits_{i=5}^{n-1}\beta_{i-2}e_{i}\Big).$
\end{center}
Thus,
\begin{center}
$\varphi([e_2,e_{n}])=\varphi(e_{4}).$
\end{center}
By Comparing the coefficients at the basis elements, we obtain that
$$\gamma_{1}=0, \ 2\gamma_n=3\alpha_1-\beta_2; \ \mbox{for} \ n\geq6 \ \mbox{we have} \ \beta_3=-3\alpha_n, \ \beta_i=0, \ 4\leq i\leq n-3.$$

By using the property of $\frac 1 2$-derivation for the products $[e_2,e_3]=0$ and $[e_3,e_n]=e_5$, we have

\begin{center}$\beta_1=0; \ \mbox{and for} \ n\geq 6 \ \mbox{we have} \ \beta_n=0, \ \beta_2=\alpha_1;$\end{center}
\begin{center}$\mbox{and for} \ n\geq 7 \ \mbox{we have} \ \beta_3=\alpha_n=0.$\end{center}

This completes the proof of the theorem. \end{proof}

Now, we will study the $\frac 1 2$-derivation of the algebra $g^3_{(n,1)}.$

\begin{theorem}\label{halfderiv3} Any $\frac 1 2$-derivations of the algebra $g^3_{(n,1)}$ has the form
\begin{itemize}
   \item for $n=7:$
  $$\begin{cases}
\varphi(e_1)=\alpha_1e_1+\alpha_2e_2+\alpha_3e_3+\alpha_4e_4+\alpha_5e_5+\alpha_6e_6, \\
\varphi(e_2)=\alpha_1e_2+4\alpha_2e_4+\beta_5e_5+\beta_6e_6,\\
\varphi(e_3)=\alpha_1e_3+2\alpha_2e_5+\frac{1}{2}(\beta_5-\alpha_3)e_6, \\
\varphi(e_4)=\alpha_1e_4+\frac{3}{2}\alpha_2e_6, \\
\varphi(e_5)=\alpha_1e_5, \ \varphi(e_6)=\alpha_1e_6, \\
\varphi(e_7)=-\alpha_2e_3-\alpha_3e_4-\alpha_4e_5+\gamma_6e_6+\alpha_1e_7, \\
  \end{cases}$$
  \item for $n\geq8:$
$$\begin{cases}
\varphi(e_1)=\alpha_1e_1+\sum\limits_{i=3}^{n-1}\alpha_ie_i, \\
\varphi(e_2)=\alpha_1e_2+\sum\limits_{i=5}^{n-3}\alpha_{i-2}e_i+\beta_{n-2}e_{n-2}+\beta_{n-1}e_{n-1}, \\
\varphi(e_3)=\alpha_1e_3+\frac{1}{2}(\beta_{n-2}-\alpha_{n-4})e_{n-1}, \\
\varphi(e_i)=\alpha_1e_i, \  4\leq i\leq n-1,\\
\varphi(e_n)=-\sum\limits_{i=4}^{n-2}\alpha_{i-1}e_i+\gamma_{n-1}e_{n-1}+\alpha_1e_n.\end{cases}$$\end{itemize}
\end{theorem}

\begin{proof}
The algebra $g^3_{(n,1)}$ has generators $e_1, e_2$ and $e_n$. Then, we have the following
\begin{center}
$\varphi(e_1)=\sum\limits_{i=1}^{n}\alpha_ie_i,\
\varphi(e_2)=\sum\limits_{i=1}^{n}\beta_ie_i,\ \varphi(e_n)=\sum\limits_{i=1}^{n}\gamma_ie_i.$
\end{center}
Now, consider the condition of $\frac 1 2$-derivation for the elements $e_2$ and $e_n$ we obtain
\begin{center}
$\varphi(e_3)=\varphi([e_1,e_2])=\frac{1}{2}([\varphi(e_1),e_2]+[e_1,\varphi(e_2)])=\frac{1}{2}([\sum\limits_{i=1}^{n}\alpha_ie_i,e_2]+[e_1,\sum\limits_{i=1}^{n}\beta_ie_i])$
\end{center}
\begin{center}
$=\frac{1}{2}\Big(\alpha_1e_3-\sum\limits_{i=3}^{n-4}\alpha_{i}e_{i+3}-\alpha_ne_4+\sum\limits_{i=2}^{n-2}\beta_{i}e_{i+1}\Big)$
\end{center}
\begin{center}
$=\frac{1}{2}\Big((\alpha_1+\beta_2)e_3+(\beta_3-\alpha_n)e_4+\beta_4e_5+\sum\limits_{i=6}^{n-1}(\beta_{i-1}-\alpha_{i-3})e_{i}\Big).$
\end{center}

We can prove the following equalities by induction:
\begin{center}
$\varphi(e_i)=\frac{1}{2^{i-2}}\Big((2^{i-2}-1)\alpha_{1}+\beta_2\Big)e_{i}+\frac{1}{2^{i-2}}\Big(\beta_{3}-(2^{i-2}-1)\alpha_n\Big)e_{i+1}$\end{center}
\begin{center}
$+\frac{1}{2^{i-2}}\Big((2^{i-2}-2)\alpha_2+\beta_{4}\Big)e_{i+2}+\frac{1}{2^{i-2}}\sum\limits_{t=i+3}^{n-1}(\beta_{t-i+2}-\alpha_{t-i})e_t$
\end{center}
for $3\leq i\leq n-4$ and
\begin{center}
$\varphi(e_{n-3})=\frac{1}{2^{n-5}}\Big((2^{n-5}-1)\alpha_{1}+\beta_2\Big)e_{n-3}+\frac{1}{2^{n-5}}\Big(\beta_{3}-(2^{n-5}-1)\alpha_n\Big)e_{n-2}$\end{center}
\begin{center}
$+\frac{1}{2^{n-5}}\Big((2^{n-5}-2)\alpha_2+\beta_{4}\Big)e_{n-1},$
\end{center}
\begin{center}
$\varphi(e_{n-2})=\frac{1}{2^{n-4}}\Big((2^{n-4}-1)\alpha_{1}+\beta_2\Big)e_{n-2}+\frac{1}{2^{n-4}}\Big(\beta_{3}-(2^{n-4}-1)\alpha_n\Big)e_{n-1},$
\end{center}
\begin{center}
$\varphi(e_{n-1})=\frac{1}{2^{n-3}}\Big((2^{n-3}-1)\alpha_{1}+\beta_2\Big)e_{n-1}.$
\end{center}

From $$0=\varphi([e_1,e_n])=\frac{1}{2}([\varphi(e_1),e_n]+[e_1,\varphi(e_n)])=\frac{1}{2}\Big([\sum\limits_{k=1}^{n}\alpha_ke_k,e_n]+[e_1,\sum\limits_{k=1}^{n}\gamma_ke_k]\Big)$$
$$=\frac{1}{2}\Big(\sum\limits_{k=2}^{n-3}\alpha_ke_{k+2}+\sum\limits_{k=2}^{n-2}\gamma_ke_{k+1}\Big)=\frac{1}{2}\Big(\sum\limits_{k=4}^{n-1}\alpha_{k-2}e_{k}+\sum\limits_{k=3}^{n-1}\gamma_{k-1}e_{k}\Big),$$ we obtain that
\begin{center}
$\gamma_2=0,\ \gamma_i=-\alpha_{i-1},\ 3\leq i\leq n-2.$
\end{center}

Now, we consider the condition of $\frac 1 2$-derivation for the elements $e_2, e_{n}:$
\begin{center}
$\varphi([e_2,e_{n}])=\frac{1}{2}([\varphi(e_2),e_{n}]+[e_2,\varphi(e_{n})])=\frac{1}{2}\Big([\sum\limits_{i=1}^{n}\beta_{i}e_i,e_{n}]+
[e_2,\sum\limits_{i=1}^{n}\gamma_{i}e_i]\Big)$
\end{center}
\begin{center}
$=\frac{1}{2}\Big(\sum\limits_{i=2}^{n-3}\beta_{i}e_{i+2}-\gamma_{1}e_3+\sum\limits_{i=3}^{n-4}\gamma_{i}e_{i+3}+\gamma_{n}e_4\Big)$
\end{center}
\begin{center}
$=\frac{1}{2}\Big(-\gamma_{1}e_3+(\beta_{2}+\gamma_{n})e_4+\beta_{3}e_5+\sum\limits_{i=6}^{n-1}(\beta_{i-2}-\gamma_{i-3})e_{i}\Big).$
\end{center}
Thus, we have
\begin{center}
$\varphi([e_2,e_{n}])=\varphi(e_{4})=\frac{1}{4}(3\alpha_{1}+\beta_2)e_{4}+\frac{1}{4}(\beta_{3}-3\alpha_n)e_{5}+\frac{1}{4}(2\alpha_2+\beta_{4})e_{6}+\frac{1}{4}\sum\limits_{t=7}^{n-1}(\beta_{t-2}-\alpha_{t-4})e_t.$
\end{center}
Compare the coefficients at the basis elements, we obtain that
$$\gamma_{1}=0, \ 2\gamma_n=3\alpha_1-\beta_2, \  \beta_3=-3\alpha_n, \ \beta_4=4\alpha_2, \  \beta_i=\alpha_{i-2}, \ 5\leq i\leq n-3.$$

By using the property of $\frac{1}{2}$-derivation for the products $[e_2,e_3]=e_6$ and $[e_3,e_n]=e_5$, we have

\begin{center}$\beta_1=0, \ \beta_n=0, \ \beta_2=\alpha_1;$\end{center}
\begin{center}$\beta_3=\alpha_n=0, \ \mbox{and for} \ n\geq 8 \ \mbox{we have} \ \beta_4=\alpha_2=0.$\end{center}

This completes the proof of the theorem. \end{proof}

The following theorems, we present an analogous descriptions of $\frac 1 2$-derivation for quasi-filiform Lie algebras of maximum length $g^1_{7}, g^2_{9}$ and $g^3_{11}$.

\begin{theorem}\label{halfderiv4} Any $\frac 1 2$-derivations of the algebra $g^1_{7}$ has the form
  $$\begin{cases}
\varphi(e_1)=\alpha_1e_1+\alpha_3e_3+\alpha_4e_4+\alpha_5e_5+\alpha_6e_6+\alpha_7e_7, \\
\varphi(e_2)=\alpha_1e_2-\frac{1}{3}\alpha_3e_4+\beta_5e_5+\beta_6e_6+\beta_7e_7,\\
\varphi(e_3)=\alpha_1e_3-\frac{2}{3}\alpha_3e_5+\frac{1}{2}(\beta_5-\alpha_4)e_6-\frac{1}{2}\alpha_5e_7, \\
\varphi(e_4)=\alpha_1e_4-\frac{1}{3}\alpha_3e_6+\frac{1}{2}\alpha_4e_7, \\
\varphi(e_5)=\alpha_1e_5-\frac{1}{2}\alpha_3e_7, \\
\varphi(e_6)=\alpha_1e_6, \\
\varphi(e_7)=\alpha_1e_7. \\
  \end{cases}$$
\end{theorem}

\begin{theorem}\label{2halfderiv12} Any $\frac 1 2$-derivations of the algebra $g^2_{9}$ has the form
  $$\begin{cases}
\varphi(e_1)=\alpha_1e_1+\alpha_5e_5+\alpha_6e_6+\alpha_7e_7+\alpha_8e_8+\alpha_9e_9, \\
\varphi(e_2)=\alpha_1e_2+\frac{1}{3}\alpha_5e_6+\beta_7e_7+\beta_8e_8+\beta_9e_9,\\
\varphi(e_3)=\alpha_1e_3-\frac{4}{3}\alpha_5e_7+\frac{1}{2}(\beta_7-5\alpha_6)e_8-\frac{1}{2}\alpha_7e_9, \\
\varphi(e_4)=\alpha_1e_4+\frac{1}{3}\alpha_5e_8+\frac{1}{2}\alpha_6e_9, \\
\varphi(e_5)=\alpha_1e_5-\frac{1}{2}\alpha_5e_9, \\
\varphi(e_i)=\alpha_1e_i, \ \ 6\leq i\leq 9. \\
  \end{cases}$$
\end{theorem}

\begin{theorem}\label{2halfderiv3} Any $\frac 1 2$-derivations of the algebra $g^3_{11}$ has the form
  $$\begin{cases}
\varphi(e_1)=\alpha_1e_1+\alpha_6e_6+\alpha_7e_7+\alpha_8e_8+\alpha_9e_9+\alpha_{10}e_{10}+\alpha_{11}e_{11}, \\
\varphi(e_2)=\alpha_1e_2-\alpha_6e_7-\alpha_7e_8+\beta_9e_9+\beta_{10}e_{10}+\beta_{11}e_{11},\\
\varphi(e_3)=\alpha_1e_3+\frac{1}{2}\beta_9e_{10}-\frac{1}{2}\alpha_9e_{11}, \\
\varphi(e_4)=\alpha_1e_4+\frac{1}{2}\alpha_7e_{10}+\frac{1}{2}\alpha_8e_{11}, \\
\varphi(e_5)=\alpha_1e_5-\frac{1}{2}\alpha_6e_{10}-\frac{1}{2}\alpha_7e_{11}, \\
\varphi(e_6)=\alpha_1e_6+\frac{1}{2}\alpha_6e_{11}, \\
\varphi(e_i)=\alpha_1e_i, \ \ 7\leq i\leq 11. \\
  \end{cases}$$
\end{theorem}

\section{Transposed Poisson structure for  quasi-filiform Lie algebras of maximum length}

In this section, we will describe all transposed Poisson algebra structures on quasi-filiform Lie algebras of maximum length.

\begin{theorem} Let $(g^1_{(5,1)}, \cdot, [-,-])$ be a transposed Poisson algebra structure defined on the Lie algebra
$g^1_{(5,1)}$. Then the multiplication of $(g^1_{(5,1)}, \cdot)$ has the following form:

${\bf TP}_1(g^1_{(5,1)}): \ e_1\cdot e_1=\alpha_{4}e_4+\alpha_{5}e_5, \ e_1\cdot e_2= \beta_{4}e_4+\beta_{5}e_5, \ e_2\cdot e_2= \beta_{9}e_4+\beta_{10}e_5;$

${\bf TP}_2(g^1_{(5,1)}): \ \begin{cases}
e_1\cdot e_1=e_3+\alpha_{4}e_4+\alpha_{5}e_5, \ e_1\cdot e_2=\beta_{3}e_3+\beta_{4}e_4+ \beta_{5}e_5, \ e_1\cdot e_3=\frac{1}{2}\beta_{3}e_4-\frac{1}{2}e_5, \\[1mm]
e_2\cdot e_2= \beta_{3}^2e_3+ \beta_{9}e_4+\beta_{10}e_5, \ e_2\cdot e_3=\frac{1}{2}\beta_{3}^2e_4-\frac{1}{2}\beta_{3}e_5;
\end{cases}$

${\bf TP}_3(g^1_{(5,1)}):\ \begin{cases}
e_1\cdot e_1= \alpha_{1}e_1+e_2+ \alpha_{3}e_3+\alpha_{4}e_4+\alpha_{5}e_5, \\[1mm]
e_1\cdot e_2=-\frac{1}{4} \left(\alpha _1-\beta _2\right){}^2e_1+\beta_{2}e_2-\frac{1}{2} \alpha _3 \left(\alpha_1-\beta _2\right)e_3+\beta_{4}e_4+\beta_{5}e_5, \\[1mm]
e_1\cdot e_3=\frac{1}{2}(\alpha_{1}+\beta_{2})e_3 -\frac{1}{4} \alpha _3 \left(\alpha _1-\beta _2\right)e_4-\frac{1}{2}\alpha_{3}e_5, \\[1mm]
e_1\cdot e_4=\frac{1}{4}(3\alpha_{1}+ \beta_{2})e_4 +\frac{1}{2}e_5, \\[1mm]
e_1\cdot e_5=-\frac{1}{8} \left(\alpha _1-\beta _2\right){}^2e_4+\frac{1}{4}(\alpha_{1}+3\beta_{2})e_5, \\[1mm]
e_2\cdot e_2= -\frac{1}{4} \left(\alpha _1-\beta _2\right){}^2 \beta _2e_1+\frac{1}{4} \left(-\alpha _1^2-2 \alpha _1 \beta _2+3 \beta _2^2\right)e_2+ \frac{1}{4} \alpha _3 \left(\alpha _1-\beta _2\right){}^2 e_3+\\
\ \ \ \ \ \ \ \ \ \ \ \ \ \ \ +\frac{1}{8} \Big(\alpha _1^2 \left(\alpha _5 \beta _2-\beta _5\right)-\beta_2(2 \alpha _4 \beta _2-\alpha _5 \beta _2^2-10 \beta _4+\beta _2 \beta_5)+\\[1mm]
\ \ \ \ \ \ \ \ \ \ \ \ \ \ \ +2 \alpha _1 \left(\alpha _4 \beta _2-\alpha _5 \beta _2^2-\beta _4+\beta _2 \beta _5\right)\Big)e_4+\\[1mm]
\ \ \ \ \ \ \ \ \ \ \ \ \ \ \ +\frac{1}{4} \left(2 \beta _4-2 \alpha _4 \beta _2-3 \alpha _5 \beta _2^2+3 \alpha _1 \left(\alpha _5 \beta _2-\beta _5\right)+7 \beta _2 \beta _5\right)e_5, \\[1mm]
e_2\cdot e_3=\frac{1}{4}(\beta_2^2-\alpha_1^2)e_3 +\frac{1}{8} \alpha _3 \left(\alpha _1-\beta _2\right){}^2e_4+\frac{1}{4}\alpha_3(\alpha_1-\beta_2)e_5, \\[1mm]
e_2\cdot e_4=\frac{1}{4} \alpha _1 \left(\beta _2-\alpha _1\right)e_4+ \frac{1}{2}\beta_{2}e_5,\\[1mm]
e_2\cdot e_5=-\frac{1}{8} \left(\alpha _1-\beta _2\right){}^2 \beta _2e_4-\frac{1}{4} \left(\alpha _1^2+\alpha _1 \beta _2-2 \beta _2^2\right)e_5, \\[1mm]
e_3\cdot e_3=\frac{1}{8}(\beta_2^2-\alpha_1^2)e_4-\frac{1}{4}(\alpha_{1}+\beta_{2})e_5;
\end{cases}$\\
where the transposed Poisson algebra has its products with respect to the bracket $[-,-]$, and the remaining products are equal to zero.

\end{theorem}
\begin{proof} Let $(g^1_{(5,1)}, \cdot, [-,-])$ be a transposed Poisson algebra structure defined on the Lie algebra $g^1_{(5,1)}$. Then
for any element of $x \in g^1_{(5,1)} $, we have the operator of multiplication $\varphi_x(y) = x \cdot y$ is a $\frac12$-derivation. Hence, for $1 \leq i \leq 5$ we derive by Theorem \ref{halfderiv1}
$$\begin{cases}
\varphi_{e_i}(e_1)=\alpha_{i,1}e_1+\alpha_{i,2}e_2+\alpha_{i,3}e_3+\alpha_{i,4}e_4+\alpha_{i,5}e_5, \\
\varphi_{e_i}(e_2)=\beta_{i,1}e_1+\beta_{i,2}e_2+\beta_{i,3}e_3+\beta_{i,4}e_4+\beta_{i,5}e_5, \\
\varphi_{e_i}(e_3)=\frac{1}{2}(\alpha_{i,1}+\beta_{i,2})e_3+\frac{1}{2}\beta_{i,3}e_4-\frac{1}{2}\alpha_{i,3}e_5, \\
\varphi_{e_i}(e_4)=\frac{1}{4}(3\alpha_{i,1}+\beta_{i,2})e_4+\frac{1}{2}\alpha_{i,2}e_5, \\
\varphi_{e_i}(e_5)=\frac{1}{2}\beta_{i,1}e_4+\frac{1}{4}(\alpha_{i,1}+3\beta_{i,2})e_5. \\
 \end{cases}$$

By considering the property of $\varphi_{e_i}(e_j)=e_i\cdot e_j=e_j\cdot e_i=\varphi_{e_j}(e_i)$, we obtain the following restrictions:

\begin{longtable}{llllll}
   $ \{e_1, e_2\},$ & $\Rightarrow$ & $\alpha_{2,1}=\beta_{1,1}, \
\alpha_{2,2}=\beta_{1,2}, \  \alpha_{2,3}=\beta_{1,3}, \ \alpha_{2,4}=\beta_{1,4}, \ \alpha_{2,5}=\beta_{1,5},$ \\

$\{e_1, e_3\},$ & $\Rightarrow$ & $\alpha_{3,1}=\alpha_{3,2}=0, \ \alpha_{3,3}=\frac{1}{2}(\alpha_{1,1}+ \beta_{1,2}), \ \alpha_{3,4}=\frac{1}{2}\beta_{1,3}, \ \alpha_{3,5}=-\frac{1}{2}\alpha_{1,3}, $ \\

   $ \{e_1, e_4\},$ & $\Rightarrow$ & $\alpha_{4,1}=\alpha_{4,2}=\alpha_{4,3}=0, \ \alpha_{4,4}=\frac{1}{4}(3\alpha_{1,1}+ \beta_{1,2}), \ \alpha_{4,5}=\frac{1}{2}\alpha_{1,2}, $ \\

    $\{e_1, e_5\},$ & $\Rightarrow$ & $\alpha_{5,1}=\alpha_{5,2}=\alpha_{5,3}=0, \ \alpha_{5,4}=\frac{1}{2}\beta_{1,1}, \ \alpha_{5,5}=\frac{1}{4}(\alpha_{1,1}+3\beta_{1,2}), $ \\

 $\{e_2, e_3\},$ & $\Rightarrow$ & $\beta_{3,1}=\beta_{3,2}=0, \  \beta_{3,3}=\frac{1}{2}(\beta_{1,1}+\beta_{2,2}), \  \beta_{3,4}=\frac{1}{2}\beta_{2,3}, \ \beta_{3,5}=-\frac{1}{2}\beta_{1,3},  $\\
 
$ \{e_2, e_4\},$ & $\Rightarrow$ & $\beta_{4,1}=\beta_{4,2}=\beta_{4,3}=0, \   \beta_{4,4}=\frac{1}{4}(3\beta_{1,1}+\beta_{2,2}), \ \beta_{4,5}=\frac{1}{2}\beta_{1,2}, $ \\

$ \{e_2, e_5\},$ & $\Rightarrow $& $\beta_{5,1}=\beta_{5,2}=\beta_{5,3}=0, \   \beta_{5,4}=\frac{1}{2}\beta_{2,1}, \ \beta_{5,5}=\frac{1}{4}(\beta_{1,1}+ 3\beta_{2,2}).$
\end{longtable}

Thus, we have the following table of commutative multiplications of the transposed Poisson algebra structure defined on the Lie algebra
$g^1_{(5,1)}$:

$$\begin{cases}
e_1\cdot e_1= \alpha_{1}e_1+\alpha_{2}e_2+ \alpha_{3}e_3+\alpha_{4}e_4+\alpha_{5}e_5, \\

e_1\cdot e_2= \beta_{1}e_1+\beta_{2}e_2+ \beta_{3}e_3+\beta_{4}e_4+\beta_{5}e_5, \\

e_1\cdot e_3=\frac{1}{2}(\alpha_{1}+ \beta_{2})e_3 +\frac{1}{2}\beta_{3}e_4-\frac{1}{2}\alpha_{3}e_5, \\

e_1\cdot e_4=\frac{1}{4}(3\alpha_{1}+ \beta_{2})e_4 +\frac{1}{2}\alpha_{2}e_5, \\

e_1\cdot e_5=\frac{1}{2}\beta_{1}e_4+\frac{1}{4}(\alpha_{1}+3\beta_{2})e_5, \\

e_2\cdot e_2= \beta_{6}e_1+\beta_{7}e_2+ \beta_{8}e_3+\beta_{9}e_4+\beta_{10}e_5, \\

e_2\cdot e_3=\frac{1}{2}(\beta_{1}+ \beta_{7})e_3+\frac{1}{2}\beta_{8}e_4-\frac{1}{2}\beta_{3}e_5, \\

e_2\cdot e_4=\frac{1}{4}(3\beta_{1}+\beta_{7})e_4+ \frac{1}{2}\beta_{2}e_5,\\

e_2\cdot e_5=\frac{1}{2}\beta_{6}e_4+\frac{1}{4}(\beta_{1}+3\beta_{7})e_5, \\
e_3\cdot e_3=\frac{1}{4}(\beta_{1}+\beta_{7})e_4-\frac{1}{4}(\alpha_{1}+\beta_{2})e_5, \\

\end{cases}$$

By using the multiplication of Lie algebra $g^1_{(5,1)}$ , we consider the general basis change and express the new basis elements $\{e'_1, e'_2, e'_3, e'_4, e'_5\}$  via the basis elements $\{e_1, e_2, e_3, e_4, e_5\}$ as follows:

$$e_1'=\sum\limits_{j=1}^{5}A_je_j, \ e_2'=\sum\limits_{j=1}^{5}B_je_j, \
e_3'=\Delta_{1,2} e_3+\Delta_{1,3}e_4+\Delta_{2,3}e_5,$$
$$e_4'=\Delta_{1,2} (A_1e_4+A_2e_5), \ e_5'=\Delta_{1,2} (B_1e_4+B_2e_5),$$
where $\Delta_{i,j}=A_iB_j-A_jB_i$ and $\Delta_{1,2}\neq0.$

Consider the multiplication $e_1'\cdot e_1'= \alpha_{1}'e_1'+\alpha_{2}'e_2'+ \alpha_{3}'e_3'+\alpha_{4}'e_4'+\alpha_{5}'e_5':$

$$e_1'\cdot e_1'=(\sum\limits_{j=1}^{5}A_je_j)\cdot (\sum\limits_{j=1}^{5}A_je_j)$$
$$=(A_1^2\alpha_{1}+ 2A_1A_2\beta_{1}+A_2^2\beta_{6})e_1+(A_1^2\alpha_{2}+ 2A_1A_2\beta_{2}+A_2^2\beta_{7})e_2+(*).$$
On the other hand, we have
$$\alpha_{1}'e_1'+\alpha_{2}'e_2'+ \alpha_{3}'e_3'+\alpha_{4}'e_4'+\alpha_{5}'e_5'=
\alpha_{1}'(\sum\limits_{j=1}^{5}A_je_j)+ \alpha_{2}'(\sum\limits_{j=1}^{5}B_je_j)+ (*)$$
$$=(A_1\alpha_{1}'+B_1\alpha_{2}')e_1+(A_2\alpha_{1}'+B_2\alpha_{2}')e_2+(*).$$

Now, we consider the multiplication $e_1'\cdot e_2'= \beta_{1}'e_1'+\beta_{2}'e_2'+ \beta_{3}'e_3'+\beta_{4}'e_4'+\beta_{5}'e_5':$

$$e_1'\cdot e_2'=(\sum\limits_{j=1}^{5}A_je_j)\cdot (\sum\limits_{j=1}^{5}B_je_j)=$$$$=(A_1B_1\alpha_{1}+ (A_1B_2+A_2B_1)\beta_{1}+A_2B_2\beta_{6})e_1+(A_1B_1\alpha_{2}+ (A_1B_2+A_2B_1)\beta_{2}+A_2B_2\beta_{7})e_2+ (*).$$
On the other hand, we get
$$\beta_{1}'e_1'+\beta_{2}'e_2'+ \beta_{3}'e_3'+\beta_{4}'e_4'+\beta_{5}'e_5'=
\beta_{1}'(\sum\limits_{j=1}^{5}A_je_j)+ \beta_{2}'(\sum\limits_{j=1}^{5}B_je_j)+ (*)$$
$$=(A_1\beta_{1}'+B_1\beta_{2}')e_1+(A_2\beta_{1}'+B_2\beta_{2}')e_2+(*).$$

Finally, we consider the multiplication $e_2'\cdot e_2'= \beta_{6}'e_1'+\beta_{7}'e_2'+ \beta_{8}'e_3'+\beta_{9}'e_4'+\beta_{10}'e_5':$

$$e_2'\cdot e_2'=(\sum\limits_{j=1}^{5}B_je_j)\cdot (\sum\limits_{j=1}^{5}B_je_j)=$$$$=(B_1^2\alpha_{1}+ 2B_1B_2\beta_{1}+B_2^2\beta_{6})e_1+ (B_1^2\alpha_{2}+ 2B_1B_2\beta_{2}+B_2^2\beta_{7})e_2+(*).$$
Thus, we obtain
$$\beta_{6}'e_1'+\beta_{7}'e_2'+ \beta_{8}'e_3'+\beta_{9}'e_4'+\beta_{10}'e_5'=
\beta_{6}'(\sum\limits_{j=1}^{5}A_je_j)+ \beta_{7}'(\sum\limits_{j=1}^{5}B_je_j)+ (*)$$
$$=(A_1\beta_{6}'+B_1\beta_{7}')e_1+(A_2\beta_{6}'+B_2\beta_{7}')e_2+(*).$$
where $(*)\in Span\langle e_3, e_4, e_5\rangle.$  By verifying all the above multiplications in the basis elements $e_1$ and $e_2$, we obtain the relations between parameters $x'=(\alpha_1',\alpha_2',\beta_1',\beta_2', \beta_6',\beta_7')$ and $x=(\alpha_1,\alpha_2,\beta_1,\beta_2, \beta_6,\beta_7)$:

\begin{equation}\label{eq1}
 \begin{cases}
A_1\alpha_{1}'+B_1\alpha_{2}'=A_1^2\alpha_{1}+ 2A_1A_2\beta_{1}+A_2^2\beta_{6},\\
A_2\alpha_{1}'+B_2\alpha_{2}'=A_1^2\alpha_{2}+ 2A_1A_2\beta_{2}+A_2^2\beta_{7}, \\
A_1\beta_{1}'+B_1\beta_{2}'=A_1B_1\alpha_{1}+ (A_1B_2+A_2B_1)\beta_{1}+A_2B_2\beta_{6},\\
A_2\beta_{1}'+B_2\beta_{2}'=A_1B_1\alpha_{2}+ (A_1B_2+A_2B_1)\beta_{2}+A_2B_2\beta_{7},\\
A_1\beta_{6}'+B_1\beta_{7}'=B_1^2\alpha_{1}+ 2B_1B_2\beta_{1}+B_2^2\beta_{6},\\
A_2\beta_{6}'+B_2\beta_{7}'=B_1^2\alpha_{2}+ 2B_1B_2\beta_{2}+B_2^2\beta_{7}.\\
 \end{cases}
\end{equation}

We can write the system \eqref{eq1} as follows:
$$x'\mathbb{A}=x\mathbb{B},$$
where $\det \mathbb{A}=\det \mathbb{B}=(\Delta_{1,2})^3.$ This means that the system \eqref{eq1} has only a zero solution $x'=\textbf{0}$ if and only if $x=\textbf{0}$.

So we have the following possibility cases.

\begin{enumerate}
    \item Let $x=\textbf{0}$. Then, from the identity $e_1\cdot (e_1\cdot e_2)=(e_1\cdot e_1)\cdot e_2$, we derive $\beta _3^2-\alpha _3 \beta _8=0$. Again, by considering the products of $e_1'\cdot e_1'=\alpha_{3}'e_3'+\alpha_{4}'e_4'+\alpha_{5}'e_5', e_1'\cdot e_2'=\beta_{3}'e_3'+\beta_{4}'e_4'+\beta_{5}'e_5'$ and  $e_2'\cdot e_2'=\beta_{8}'e_3'+\beta_{9}'e_4'+\beta_{10}'e_5'$, we have the following expressions:
    $$\begin{array}{l}
    \alpha_3'=\frac{A_1^2\alpha_3+2A_1A_2\beta_3+A_2^2\beta_8}{\Delta_{1,2}},\\[1mm]
    \beta_3'=\frac{A_1B_1\alpha_3+(A_1B_2+A_2B_1)\beta_3+A_2B_2\beta_8}{\Delta_{1,2}},\\[1mm]
    \beta_8'=\frac{B_1^2\alpha_3+2B_1B_2\beta_3+B_2^2\beta_8}{\Delta_{1,2}}.\end{array}$$
From the above expressions we can conclude that $y=(\alpha_3,\beta_3,\beta_8)$ is zero invariant.
 \begin{enumerate}
     \item If $y=\textbf{0}$, then we have the algebra ${\bf TP}_1(g^1_{(5,1)}).$
     \item If $y\neq\textbf{0}$, then  without loss of generality, we assume that $\alpha_3'=1$. Thus, we derive $\beta_8=\beta_3^2$ and obtain the algebra ${\bf TP}_2(g^1_{(5,1)}).$
 \end{enumerate}

\item Let $x\neq\textbf{0}$. Then, the system \eqref{eq1} can be written as follows:
$$\begin{array}{l}
\alpha_1'=\frac{A_1^2B_2\alpha_1+2A_1A_2B_2\beta_1+ A_2^2B_2\beta_6-A_1^2B_1\alpha_2-2A_1A_2B_1\beta_2-A_2^2B_1\beta_7}{\Delta_{1,2}},\\[1mm]
\alpha_2'=\frac{A_1^2A_2\alpha_1+2A_1A_2^2\beta_1+ A_2^3\beta_6-A_1^3\alpha_2-2A_1^2A_2\beta_2-A_1A_2^2\beta_7}{-\Delta_{1,2}},\\[1mm]
\beta_1'=\frac{A_1B_1B_2\alpha_1+(A_1B_2+A_2B_1)B_2\beta_1+A_2B_2^2\beta_6-A_1B_1^2\alpha_2-(A_1B_2+A_2B_1)B_1\beta_2-A_2B_1B_2\beta_7}{\Delta_{1,2}},\\[1mm]
\beta_2'=\frac{A_1A_2B_1\alpha_1+(A_1B_2+A_2B_1)A_2\beta_1+A_2^2B_2\beta_6-A_1^2B_1\alpha_2-(A_1B_2+A_2B_1)A_1\beta_2-A_1A_2B_2\beta_7}{-\Delta_{1,2}},\\[1mm]
\beta_6'=\frac{B_1^2B_2\alpha_1+2B_1B_2^2\beta_1+ B_2^3\beta_6-B_1^3\alpha_2-2B_1^2B_2\beta_2-B_1B_2^2\beta_7}{\Delta_{1,2}},\\[1mm]
\beta_7'=\frac{A_2B_1^2\alpha_1+2A_2B_1B_2\beta_1+ A_2B_2^2\beta_6-A_1B_1^2\alpha_2-2A_1B_1B_2\beta_2-A_1B_2^2\beta_7}{-\Delta_{1,2}}.\end{array}$$
It can be shown that $\alpha_2'=1$ and by considering the associative identity $x\cdot (y\cdot z)=(x\cdot y)\cdot z$, we obtain the following restrictions on structure constants:
$$\begin{array}{lll}
\{e_5,e_1,e_1\}, & \Rightarrow & \beta _6=\beta _1 \beta _2, \ \beta _7=\frac{1}{4} \left(3 \beta _2^2-2 \alpha _1 \beta _2-\alpha _1^2\right), \\[1mm]
\{e_1,e_1,e_3\}, & \Rightarrow & \beta _1=-\frac{1}{4} \left(\alpha _1-\beta_2\right)^2, \ \beta_3=\frac{1}{2}\alpha_3 \left(\beta_2-\alpha_1\right),\  \beta_8= \frac{1}{4}\alpha _3\left(\alpha _1-\beta _2\right)^2,\\[1mm]
\{e_2,e_1,e_1\}, & \Rightarrow & \beta _{10}= \frac{1}{4} \Big(-2 \alpha _4 \beta _2-3 \alpha _5 \beta _2^2+2 \beta _4+3 \alpha _1 \left(\alpha _5 \beta _2-\beta _5\right)+7 \beta _2 \beta _5\Big),  \\[1mm]
&  & \beta _9=\frac{1}{8} \Big(\alpha _1^2 \left(\alpha _5 \beta _2-\beta _5\right)+\beta _2 \left(-2 \alpha _4 \beta _2+\alpha _5 \beta _2^2+10 \beta _4-\beta _2 \beta _5\right)+\\[1mm]
 &  &  \ \ \ \ \ \ \ +2 \alpha _1 \left(\alpha _4 \beta _2-\alpha _5 \beta _2^2-\beta _4+\beta _2 \beta _5\right)\Big). \\[1mm]
\end{array}$$
Hence, the algebra ${\bf TP}_3(g^1_{(5,1)})$  is obtained.

\end{enumerate}

\end{proof}

\begin{theorem} Let $(g^1_{(n,1)}, \cdot, [-,-])$ be a transposed Poisson algebra structure defined on the Lie algebra
$g^1_{(n,1)}$ and $n\geq 7$. Then, the multiplication of $(g^1_{(n,1)}, \cdot)$ has the following form:

${\bf TP}_1(g^1_{(n,1)}):\ \begin{cases}
e_1\cdot e_1=\sum\limits_{j=4}^{n}\alpha_{j}e_j, \ e_1\cdot e_2=\beta_{1}e_{n-2}+\beta_{2}e_{n-1}+\beta_{3}e_n, \\[1mm]
e_1\cdot e_3=\frac{1}{2}\beta_{1}e_{n-1}-\frac{1}{2}\alpha_{n-2}e_{n}, \ e_1\cdot e_j=\frac{(-1)^j}{2}\alpha_{n-j+1}e_n, \  4\leq j\leq n-3, \\[1mm]
e_2\cdot e_2=\beta_{4}e_{n-2}+\beta_{5}e_{n-1}+\beta_{6}e_n, \ e_2\cdot e_3=\frac{1}{2}\beta_{4}e_{n-1}-\frac{1}{2}\beta_{1}e_{n};
\end{cases}
$

${\bf TP}_2(g^1_{(n,1)}):\ \begin{cases}
e_1\cdot e_1=e_3+ \sum\limits_{j=4}^{n}\alpha_{j}e_j, \
e_1\cdot e_2=\beta_{1}e_{n-2}+\beta_{2}e_{n-1}+\beta_{3}e_n, \\[1mm]
e_1\cdot e_3=\frac{1}{2}\beta_{1}e_{n-1}-\frac{1}{2}\alpha_{n-2}e_{n}, \
e_1\cdot e_j=\frac{(-1)^j}{2}\alpha_{n-j+1}e_n, \  4\leq j\leq n-3, \\[1mm]
e_1\cdot e_{n-2}=-\frac{1}{2}e_n, \
e_2\cdot e_2=\beta_{4}e_{n-1}+\beta_{5}e_n, \
e_2\cdot e_3=-\frac{1}{2}\beta_{1}e_{n};
\end{cases}
$

${\bf TP}_3(g^1_{(n,1)}):\ \begin{cases}
e_1\cdot e_1=e_2+ \sum\limits_{j=3}^{n}\alpha_{j}e_j, \
e_1\cdot e_2=2\beta_{1}e_{n-1}+\beta_{2}e_n, \\[1mm]
e_1\cdot e_3=-\frac{1}{2}\alpha_{n-2}e_{n}, \
e_1\cdot e_j=\frac{(-1)^j}{2}\alpha_{n-j+1}e_n, \  4\leq j\leq n-2, \\
e_1\cdot e_{n-1}=\frac{1}{2}e_n, \
e_2\cdot e_2=\beta_{1}e_n;
\end{cases}
$\\
where the transposed Poisson algebra has its products with respect to the bracket $[-,-]$, and the remaining products are equal to zero.

\end{theorem}
\begin{proof} Let $(g^1_{(n,1)}, \cdot, [-,-])$ be a transposed Poisson algebra structure defined on the Lie algebra $g^1_{(n,1)}$. Then,
for any element $x \in g^1_{(n,1)} $, we have the operator of multiplication $\varphi_x(y) = x \cdot y$ is a $\frac12$-derivation. Hence, for $1 \leq i \leq n$ we derive by Theorem \ref{halfderiv1}:

$$\begin{cases}
\varphi_{e_i}(e_1)=\sum\limits_{j=1}^{n}\alpha_{i,j}e_j, \\
\varphi_{e_i}(e_2)=\alpha_{i,1}e_2+\beta_{i,n-2}e_{n-2}+\beta_{i,n-1}e_{n-1}+\beta_{i,n}e_n, \\
\varphi_{e_i}(e_3)=\alpha_{i,1}e_3+\frac{1}{2}\beta_{i,n-2}e_{n-1}-\frac{1}{2}\alpha_{i,n-2}e_{n}, \\
\varphi_{e_i}(e_j)=\alpha_{i,1}e_j+\frac{(-1)^j}{2}\alpha_{i,n-j+1}e_n, \  4\leq j\leq n-1, \\
\varphi_{e_i}(e_n)=\alpha_{i,1}e_n. \\
\end{cases}$$

By considering equality  $\varphi_{e_i}(e_j)=e_i\cdot e_j=e_j\cdot e_i=\varphi_{e_j}(e_i)$, we have following restrictions:

\begin{longtable}{lllllllll}
    $ \{e_1, e_2\},$ & $\Rightarrow$ & $\alpha_{2,1}=0, \ \alpha_{2,2}=\alpha_{1,1}, \ \alpha_{2,t}=0,\ 3\leq t\leq n-3,$ \\
& & $\alpha_{2,n-2}=\beta_{1,n-2}, \ \alpha_{2,n-1}=\beta_{1,n-1}, \ \alpha_{2,n}=\beta_{1,n},$ \\

$ \{e_1, e_3\},$ & $\Rightarrow $ & $\alpha_{3,1}=\alpha_{3,2}=0, \ \alpha_{3,3}=\alpha_{1,1},  \ \alpha_{3,t}=0,\ 4\leq t\leq n-2,$ \\

& & $\alpha_{3,n-1}=\frac{1}{2}\beta_{1,n-2}, \ \alpha_{3,n}=-\frac12\beta_{1,n},$ \\

$ \{e_1, e_j\},$ & $\Rightarrow$ & $\alpha_{j,t}=0,\ 1\leq t\leq j-1, \ \alpha_{j,j}=\alpha_{1,1},$ \\ 
$4\leq j\leq n-1,$ & & $\alpha_{j,n}=\frac{(-1)^j}{2}\alpha_{1,n-j+1}, \  \alpha_{j,t}=0,\ j+1\leq t\leq n-1,$ \\

$\{e_1, e_n\},$ & $\Rightarrow$ & $\alpha_{n,t}=0,\ 1\leq t\leq n-1, \ \alpha_{n,n}=\alpha_{1,1}, $ \\

$\{e_2, e_3\},$ & $\Rightarrow$ & $\beta_{3,n-2}=0, \ \beta_{3,n-1}=\frac{1}{2}\beta_{2,n-2}, \ \beta_{3,n}=-\frac{1}{2}\beta_{1,n-2},$  \\
$ \{e_2, e_j\},$ & $\Rightarrow$ & $\beta_{j,n-2}=\beta_{j,n-1}=\beta_{j,n}=0, $ \\
$ 4\leq j\leq n-1,$ & & \\
$ \{e_2, e_n\},$ & $\Rightarrow$ & $\beta_{n,n-2}=\beta_{n,n-1}=\beta_{n,n}=0,$\\
 $ \{e_3, e_j\},$ & $\Rightarrow$ & $\alpha_{3,t}=0, \ 2\leq t\leq n-3, $\\
$ 4\leq j\leq n-1.$& &\\
\end{longtable}

Thus, we have the following table of commutative multiplications of the transposed Poisson algebra structure $(g^1_{(n,1)},, \cdot, [-,-])$:

\begin{equation}\label{gn1}
\begin{cases}
e_1\cdot e_1=\sum\limits_{j=2}^{n}\alpha_{j}e_j, \\
e_1\cdot e_2=\beta_{1}e_{n-2}+\beta_{2}e_{n-1}+\beta_{3}e_n, \\
e_1\cdot e_3=\frac{1}{2}\beta_{1}e_{n-1}-\frac{1}{2}\alpha_{n-2}e_{n}, \\
e_1\cdot e_j=\frac{(-1)^j}{2}\alpha_{n-j+1}e_n, \  4\leq j\leq n-1, \\
e_2\cdot e_2=\beta_{4}e_{n-2}+\beta_{5}e_{n-1}+\beta_{6}e_n, \\
e_2\cdot e_3=\frac{1}{2}\beta_{4}e_{n-1}-\frac{1}{2}\beta_{1}e_{n}.\\
\end{cases}
\end{equation}

Consider the associative identity $x\cdot (y\cdot z)=(x\cdot y)\cdot z$ in \eqref{gn1}, we obtain the following restrictions on structure constants:
$$\begin{array}{lll}
\{e_1,e_2,e_3\}, & \Rightarrow & \alpha_{2}\beta_{4}=0,  \\[1mm]
\{e_1,e_2,e_2\}, & \Rightarrow & \alpha_{3}\beta_{4}-\alpha_{2}\beta_{5}=0,  \\[1mm]
\{e_1,e_1,e_3\}, & \Rightarrow & \alpha_{2}\beta_{1}=0,  \\[1mm]
\{e_1,e_1,e_2\}, & \Rightarrow & 2\alpha_{2}\beta_{5}+\alpha_{3}\beta_{4}=0, \ \alpha_{2}(2\beta_{6}-\beta_{2})=0,  \\[1mm]
\{e_1,e_1,e_i\}, & \Rightarrow & \alpha_{2}\beta_{4}=0, \ \alpha_{3}\beta_{4}=0, \ \alpha_{2}\beta_{5}=0,\ \alpha_{2}\beta_{1}=0, \  \alpha_{2}(2\beta_{6}-\beta_{2})=0, \\[1mm]
4\leq i\leq n-1.& &
\end{array}$$

Now we consider the general change of basis:
%

$$e_1'=\sum\limits_{j=1}^{n}A_je_j, \ e_2'=\sum\limits_{j=1}^{n}B_je_j, \ e_{i+1}'=[e_1',e_{i}'], \ 2\leq i\leq n-2, \ e_n'=[e_2',e_{n-2}'],$$
where $B_1=0$ and $A_1B_2\neq0$. Then, from the multiplication $e_1\cdot e_1=\sum\limits_{j=2}^{n}\alpha_{j}e_j,$ we discover that the structure constant of the $\alpha_2$ changes as follows:
$$\alpha_2'=\frac{A_1^2}{B_2}\alpha_2.$$

Thus, we have the following cases.
\begin{enumerate}
    \item If $\alpha_{2}=0.$  Again, by using a change of basis and from the product  $e_1\cdot e_1=\sum\limits_{j=3}^{n}\alpha_{j}e_j,$ we obtain the following relation:
$$\alpha_3'=\frac{A_1}{B_2}\alpha_3.$$

Now, we consider the following subcases.
\begin{enumerate}

 \item If $\alpha_{3}=0.$ Then, we have the algebra ${\bf TP}_1(g^1_{(n,1)}).$

 \item If $\alpha_{3}\neq 0,$ then, we have $\beta_{4}=0,$ and via transformation
$$\phi(e_1)=e_1, \ \phi(e_i)=\alpha_{3}e_i, \ 2\leq i\leq n-1,  \ \phi(e_{n})=\alpha_{3}^2e_n,$$
we get the algebra ${\bf TP}_2(g^1_{(n,1)})$.
\end{enumerate}

\item If $\alpha_{2}\neq0,$ then we derive $\beta_{4}=\beta_{5}=\beta_{1}=0, \ \beta_{2}=2\beta_{6},$ and we via transformation
$$\phi(e_1)=e_1, \ \phi(e_i)=\alpha_{2}e_i, \ 2\leq i\leq n-1,  \ \phi(e_{n})=\alpha_{2}^2e_n,$$
we obtain the algebra  ${\bf TP}_3(g^1_{(n,1)})$.

\end{enumerate}

\end{proof}

\begin{theorem} Let $(g^2_{(5,1)}, \cdot, [-,-])$ be a transposed Poisson algebra structure defined on the Lie algebra
$g^2_{(5,1)}$. Then the multiplication of $(g^2_{(5,1)}, \cdot)$ has the following form:

 ${\bf TP}_1(g^2_{(5,1)}): \ \begin{cases}
e_1\cdot e_1=e_2+\alpha_{3}e_3+\alpha_{4}e_4+\alpha_{5}e_5, \ e_1\cdot e_2=(\alpha_5(1-2\alpha_{16})-2 \alpha_{15})e_3+\alpha_{8}e_4,\\[1mm]
e_1\cdot e_3=-(\alpha_{15}+\alpha_5\alpha _{16})e_4, \ e_1\cdot e_5=-e_3+\alpha_{10}e_4, \\[1mm]
e_2\cdot e_2=(2(\alpha_{15}+\alpha_5\alpha _{16})^2-\alpha_5(2\alpha_{15}+\alpha _{16}))e_4,\ e_2\cdot e_5=\alpha_{15}e_4, \ e_5\cdot e_5=\alpha_{16}e_4;
\end{cases}$

 ${\bf TP}_2(g^2_{(5,1)}):\ \begin{cases}
e_1\cdot e_1=\alpha_{3}e_3+\alpha_{4}e_4, \ e_1\cdot e_2=\alpha_{8}e_4,\ e_1\cdot e_5=\alpha_{10}e_4, \\[1mm]
e_2\cdot e_2=\alpha_{13}e_4,\ e_2\cdot e_5=\alpha_{15}e_4, \ e_5\cdot e_5=e_4; \end{cases}$

 ${\bf TP}_3(g^2_{(5,1)}):\ \begin{cases}
e_1\cdot e_1=\alpha_{7}^2e_3+ \alpha_{4}e_4, \ e_1\cdot e_2=\alpha_{7}e_3+ \alpha_{8}e_4,\ e_1\cdot e_3=\frac{1}{2}\alpha_{7}e_4, \ e_1\cdot e_5=\alpha_{10}e_4, \\[1mm]
e_2\cdot e_2=e_3+\alpha_{13}e_4,\ e_2\cdot e_3=\frac{1}{2}e_4, \ e_2\cdot e_5=\alpha_{15}e_4, \ e_5\cdot e_5=e_4; \end{cases}$

 ${\bf TP}_4(g^2_{(5,1)}):\ \begin{cases}
e_1\cdot e_1=\frac{\left(\alpha _7^2-\alpha _5 \left(\alpha _7+2 \alpha _{15}\right)\right)}{\alpha _{12}}e_3+ \alpha_{4}e_4+ \alpha_{5}e_5, \\[1mm]
e_1\cdot e_2=\alpha_{7}e_3+ \alpha_{8}e_4+ \alpha_{9}e_5,\ e_1\cdot e_3=\frac{1}{2}(\alpha_{7}-\alpha_{5})e_4, \ e_1\cdot e_5=\frac{1}{2} \alpha _5 \alpha _{12}e_4, \\[1mm]
e_2\cdot e_2=\alpha_{12}e_3+ \alpha_{13}e_4+e_5,\ e_2\cdot e_3=\frac{1}{2}\alpha_{12}e_4, \ e_2\cdot e_5=\alpha_{15}e_4;\end{cases}$

 ${\bf TP}_5(g^2_{(5,1)}):\ \begin{cases}
e_1\cdot e_1=\alpha_{3}e_3+\alpha_{4}e_4+ \alpha_{5}e_5, \ e_1\cdot e_2=\alpha_{7}e_3+ \alpha_{8}e_4,\\[1mm]
e_1\cdot e_3=\frac{1}{2}(\alpha_{7}-\alpha_{5})e_4, \ e_2\cdot e_2=\alpha_{13}e_4+e_5,\ e_2\cdot e_5=\frac{\alpha _7 \left(\alpha _7-\alpha _5\right)}{2 \alpha _5}e_4;
\end{cases}$

${\bf TP}_6(g^2_{(5,1)}):\
e_1\cdot e_1=\alpha_{3}e_3+\alpha_{4}e_4, \ e_1\cdot e_2=\alpha_{8}e_4,\ e_2\cdot e_2= \alpha_{13}e_4+e_5,\ e_2\cdot e_5=\alpha_{15}e_4;
$

${\bf TP}_7(g^2_{(5,1)}):\ \begin{cases}
e_1\cdot e_1=\alpha_{3}e_3+ \alpha_{4}e_4, \ e_1\cdot e_2=2\alpha_{15}e_3+\alpha_{8}e_4+e_5,\\[1mm]
e_1\cdot e_3=\frac{1}{2}\alpha_{15}e_4, \ e_1\cdot e_5=\frac{1}{2} \left(\alpha _3 \left(\alpha _{12}-1\right)-4 \alpha _{15}^2\right)e_4, \\[1mm]
e_2\cdot e_2=\alpha_{12}e_3+ \alpha_{13}e_4,\ e_2\cdot e_3=\frac{1}{2}(\alpha_{12}-1)e_4, \
e_2\cdot e_5=\alpha_{15}e_4; \end{cases}$

${\bf TP}_8(g^2_{(5,1)}):\ \begin{cases}
e_1\cdot e_1=\alpha_{7}^2e_3+\alpha_{4}e_4, \ e_1\cdot e_2=\alpha_{7}e_3+ \alpha_{8}e_4,\ e_1\cdot e_3=\frac{1}{2}\alpha_{7}e_4, \\[1mm]
e_1\cdot e_5=\alpha_{10}e_4, \ e_2\cdot e_2=e_3+\alpha_{13}e_4,\ e_2\cdot e_3=\frac{1}{2}e_4, \ e_2\cdot e_5=\alpha_{15}e_4, \end{cases}$

${\bf TP}_9(g^2_{(5,1)}):\ \begin{cases}
e_1\cdot e_1=\alpha_{3}e_3+ \alpha_{4}e_4+ \alpha_{5}e_5, \ e_1\cdot e_2=\alpha_{7}e_3+ \alpha_{8}e_4,\ e_1\cdot e_3=\frac{1}{2}(\alpha_{7}-\alpha_{5})e_4, \\[1mm]
e_1\cdot e_5=\alpha_{10}e_4, \ e_2\cdot e_2=\alpha_{13}e_4,\
e_2\cdot e_5=\frac{\alpha _7 \left(\alpha_7-\alpha _5\right)}{2\alpha _5}e_4;
\end{cases}$

${\bf TP}_{10}(g^2_{(5,1)}):\
e_1\cdot e_1=\alpha_{3}e_3+ \alpha_{4}e_4, \ e_1\cdot e_2=\alpha_{8}e_4,\ e_1\cdot e_5=\alpha_{10}e_4, \ e_2\cdot e_2= \alpha_{13}e_4,\ e_2\cdot e_5=\alpha_{15}e_4;
$\\
where the transposed Poisson algebra has its products with respect to the bracket $[-,-]$, and remaining products are equal to zero.
\end{theorem}

\begin{proof}
Let $(g^2_{(5,1)}, \cdot, [-,-])$ be a transposed Poisson algebra structure defined on Lie algebra $g^2_{(5,1)}$. Then,
for any element of $x \in g^2_{(5,1)} $, we have operator of multiplication $\varphi_x(y) = x \cdot y$ is a $\frac12$-derivation. Hence, by using Theorem \ref{halfderiv2} $1 \leq i\leq  n$, we derive

$$\begin{cases}
\varphi_{e_i}(e_1)=\alpha_{i,1}e_1+\alpha_{i,2}e_2+ \alpha_{i,3}e_3+\alpha_{i,4}e_4+\alpha_{i,5}e_5, \\
\varphi_{e_i}(e_2)=\beta_{i,2}e_2+\beta_{i,3}e_3+ \beta_{i,4}e_4+\beta_{i,5}e_5,\\
\varphi_{e_i}(e_3)=\frac{1}{2}(\alpha_{i,1}+\beta_{i,2})e_3+\frac{1}{2}(\beta_{i,3}-\alpha_{i,5})e_4, \\
\varphi_{e_i}(e_4)=\frac{1}{4}(3\alpha_{i,1}+\beta_{i,2})e_4, \\
\varphi_{e_i}(e_5)=-\alpha_{i,2}e_3+\gamma_{i,4}e_4 +\frac{1}{2}(3\alpha_{i,1}-\beta_{i,2})e_5, \\
\end{cases}$$

By considering the property of $\varphi_{e_i}(e_j)=e_i\cdot e_j=e_j\cdot e_i=\varphi_{e_j}(e_i)$, we obtain the following restrictions:

 $$\begin{array}{lll}
    \{e_1, e_2\}, & \Rightarrow & \alpha_{2,1}=0, \  \alpha_{2,2}=\beta_{1,2}, \  \alpha_{2,3}=\beta_{1,3}, \ \alpha_{2,4}=\beta_{1,4}, \ \alpha_{2,5}=\beta_{1,5},\\[1mm]
    \{e_1, e_3\}, & \Rightarrow & \alpha_{3,1}=\alpha_{3,2}=0, \  \alpha_{3,3}=\frac{1}{2}(\alpha_{1,1}+\beta_{1,2}), \  \alpha_{3,4}=\frac{1}{2}(\beta_{1,3}-\alpha_{1,5}), \ \alpha_{3,5}=0,  \\[1mm]
    \{e_1, e_4\}, & \Rightarrow & \alpha_{4,1}=\alpha_{4,2}=\alpha_{4,3}=0, \  \alpha_{4,4}=\frac{1}{4}(3\alpha_{1,1}+\beta_{1,2}), \ \alpha_{4,5}=0,  \\[1mm]
    \{e_1, e_5\}, & \Rightarrow & \alpha_{5,1}=\alpha_{5,2}=0, \
\alpha_{5,3}=-\alpha_{1,2}, \  \alpha_{5,4}=\gamma_{1,4}, \ \alpha_{5,5}=\frac{1}{2}(3\alpha_{1,1}-\beta_{1,2}),  \\[1mm]
 \{e_2, e_3\}, & \Rightarrow & \beta_{3,2}=0, \  \beta_{3,3}=\frac{1}{2}\beta_{2,2}, \  \beta_{3,4}=\frac{1}{2}(\beta_{2,3}-\beta_{1,3}), \ \beta_{3,5}=0,  \\[1mm]
 \{e_2, e_4\}, & \Rightarrow & \beta_{4,2}=\beta_{4,3}=0, \  \beta_{4,4}=\frac{1}{4}\beta_{2,2}, \ \beta_{4,5}=0,  \\[1mm]
 \{e_2, e_5\}, & \Rightarrow & \beta_{5,2}=0, \
\beta_{5,3}=-\beta_{1,2}, \  \beta_{5,4}=\gamma_{2,4}, \ \beta_{5,5}=-\frac{1}{2}\beta_{2,2},\\[1mm]
\{e_3, e_5\}, & \Rightarrow & \gamma_{3,4}=-\frac{1}{4}(3\alpha_{1,1}+\beta_{1,2}),\\[1mm]
\{e_4, e_5\}, & \Rightarrow & \gamma_{4,4}=0.\\[1mm]
    \end{array}$$

Thus, we have the following table of commutative multiplications of the transposed Poisson algebra structure defined on  $g^2_{(5,1)}$:

$$\begin{cases}
e_1\cdot e_1=\alpha_{1}e_1+\alpha_{2}e_2+ \alpha_{3}e_3+\alpha_{4}e_4+\alpha_{5}e_5, \ e_1\cdot e_2=\alpha_{6}e_2+\alpha_{7}e_3+ \alpha_{8}e_4+\alpha_{9}e_5,\\[1mm]
e_1\cdot e_3=\frac{1}{2}(\alpha_{1}+\alpha_{6})e_3 +\frac{1}{2}(\alpha_{7}-\alpha_{5})e_4, \ e_1\cdot e_4=\frac{1}{4}(3\alpha_{1}+\alpha_{6})e_4, \\[1mm]
e_1\cdot e_5=-\alpha_{2}e_3+\alpha_{10}e_4 +\frac{1}{2}(3\alpha_{1}-\alpha_{6})e_5, \ e_2\cdot e_2=\alpha_{11}e_2+\alpha_{12}e_3+ \alpha_{13}e_4+\alpha_{14}e_5,\\[1mm]
e_2\cdot e_3=\frac{1}{2}\alpha_{11}e_3+\frac{1}{2}(\alpha_{12}-\alpha_{9})e_4, \ e_2\cdot e_4=\frac{1}{4}\alpha_{11}e_4, \ e_2\cdot e_5=-\alpha_{6}e_3+\alpha_{15}e_4 +\frac{1}{2}\alpha_{11}e_5, \\[1mm]
e_3\cdot e_3=\frac{1}{4}\alpha_{11}e_4, \ e_3\cdot e_5=-\frac14(3\alpha_{1}+\alpha_{6})e_4, \ e_5\cdot e_5=\alpha_{16}e_4.
\end{cases}$$

Now, consider the associative identity $x\cdot (y\cdot z)=(x\cdot y)\cdot z$, we obtain the following restrictions on structure constants:
$$\begin{array}{lll}
\{e_4,e_2,e_2\}, & \Rightarrow & \alpha_{11}=0,  \\[1mm]
\{e_5,e_1,e_3\}, & \Rightarrow & (\alpha_1-\alpha_6)(3\alpha_1+\alpha_6)=0,  \\[1mm]
\{e_5,e_2,e_1\}, & \Rightarrow & \alpha_6(\alpha_1-\alpha_6)=0,  \\[1mm]
\{e_5,e_5,e_2\}, & \Rightarrow & \alpha_6(3\alpha_1+\alpha_6)=0.
\end{array}$$

Hence, we can conclude that $\alpha_1=\alpha_6=0$ and the table of multiplication of the algebra will have the following form:

$$\begin{cases}
e_1\cdot e_1=\alpha_{2}e_2+ \alpha_{3}e_3+ \alpha_{4}e_4+\alpha_{5}e_5, \ e_1\cdot e_2=\alpha_{7}e_3+ \alpha_{8}e_4+ \alpha_{9}e_5,\\[1mm]
e_1\cdot e_3=\frac{1}{2}(\alpha_{7}-\alpha_{5})e_4, \ e_1\cdot e_5=-\alpha_{2}e_3+\alpha_{10}e_4, \ e_2\cdot e_2=\alpha_{12}e_3+ \alpha_{13}e_4+\alpha_{14}e_5,\\[1mm]
e_2\cdot e_3=\frac{1}{2}(\alpha_{12}-\alpha_{9})e_4, \ e_2\cdot e_5=\alpha_{15}e_4, \ e_5\cdot e_5=\alpha_{16}e_4.
\end{cases}$$

 By considering  the general change of the generators of basis
$$e_1'=\sum\limits_{t=1}^{5}A_te_t, \ e_2'=\sum\limits_{t=1}^{5}B_te_t, \ e_5'=\sum\limits_{t=1}^{5}C_te_t,$$ and from the  product $e_1'\cdot e_1'=\alpha_{2}'e_2'+\alpha_{3}'e_3'+\alpha_{4}'e_4'+\alpha_{5}'e_5'$, we derive
$$\alpha_2'=\frac{A_1^2}{B_2}\alpha_2.$$
We have the following possible cases.
\begin{enumerate}
    \item Let $\alpha_2\neq0$. Then by choosing $B_2=A_1^2\alpha_2$, we get $\alpha_2'=1.$ By considering the associative identity $x\cdot (y\cdot z)=(x\cdot y)\cdot z$, we obtain the restrictions on structure constants as follows:
$$\begin{array}{lll}
\{e_1,e_1,e_3\}, & \Rightarrow & \alpha_{12}=\alpha_{9},  \\[1mm]
\{e_1,e_1,e_2\}, & \Rightarrow & \alpha_{9}=\alpha_{14}=0,  \\[1mm]
\{e_1,e_1,e_5\}, & \Rightarrow & \alpha _7=\alpha _5 \left(1-2 \alpha _{16}\right)-2 \alpha _{15},  \\[1mm]
\{e_2,e_1,e_1\}, & \Rightarrow & \alpha _{13}=2 \alpha _{15}^2+2 \alpha _5 \alpha _{15} \left(-1+2 \alpha _{16}\right)+\alpha _5^2 \alpha _{16} \left(-1+2 \alpha _{16}\right).\\[1mm]
\end{array}$$
In this case, we have the algebra ${\bf TP}_1(g^2_{(5,1)})$.

\item Let $\alpha_2=0$. Then, from the product $e_5'\cdot e_5'=\alpha_{16}'e_4'$, we obtain the following expression:
$$\alpha_{16}'=\frac{A_1^2}{B_2}\alpha_{16}.$$

\begin{enumerate}
    \item  If $\alpha_{16}\neq0$, then by choosing $B_2=A_1^2\alpha_{16}$, we obtain $\alpha_{16}'=1.$ Considering the associative identity, we obtain the following restrictions:
$$\begin{array}{lll|lll}
\{e_1,e_1,e_5\}, & \Rightarrow & \alpha_{5}=0,  &
\{e_1,e_2,e_5\}, & \Rightarrow & \alpha_{9}=0, \\
\{e_5,e_2,e_2\}, & \Rightarrow & \alpha_{14}=0, &
\{e_2,e_1,e_1\}, & \Rightarrow & \alpha_{7}^2=\alpha_3\alpha_{12}.  \\
[1mm]
\end{array}$$

Again, using the change of basis for the product $e_2'\cdot e_2'=\alpha_{12}'e_3'+ \alpha_{13}'e_4'$, we will get the following  expression:
$$\alpha_{12}'=A_1\alpha_{12}.$$

   If $\alpha_{12}=0,$ then we get $\alpha_7=0$ and obtain the
algebra   ${\bf TP}_2(g^2_{(5,1)}).$

If $\alpha_{12}\neq0,$ then by choosing $A_1=\frac{1}{\alpha_{12}}$, we can assume $\alpha_{12}'=1, \alpha_3=\alpha_7^2.$  Therefore, the
algebra ${\bf TP}_3(g^2_{(5,1)})$  is obtained.

\item If $\alpha_{16}=0,$ then apply the change of basis we get
$$\alpha_{14}'=\frac{B_2^2}{A_1^2}\alpha_{14}, \ \alpha_{9}'=\frac{B_2(A_1\alpha_9+A_2\alpha_{14})}{A_1^2}, \ \alpha_{12}'=\frac{B_2(A_1\alpha_{12}+A_2\alpha_{14})}{A_1^2}.$$

\begin{enumerate}
    \item If $\alpha_{14}\neq 0$ then  by choosing $A_1^2=B_1^2\alpha_{14}, \ A_2=-\frac{A_1\alpha_9}{\alpha_{14}}$, we obtain $\alpha_9'=0, \ \alpha_{14}'=1.$
    Again, using the change of basis elements of algebra, we have the following  expression:
    $$\alpha_5'=\alpha_5, \quad  \quad \alpha_{12}'=\frac{B_2}{A_1}\alpha_{12},  \quad B_2^2=A_1^2.$$

    By considering the associative identity for the triples $\{e_1,e_1,e_2\}$ va $\{e_1,e_2,e_2\}$, we get
    $$2 \alpha _{10}=\alpha _5 \alpha _{12}, \ \alpha _5 \alpha _7-\alpha _7^2+\alpha _3 \alpha _{12}+2 \alpha _5 \alpha _{15}=0,$$
 respectively. we consider following cases.

$\alpha_{12}\neq0,$ then we have the algebra   ${\bf TP}_4(g^2_{(5,1)}).$

$\alpha_{12}=0$ and $\alpha_5\neq0,$ then we have the algebra   ${\bf TP}_5(g^2_{(5,1)}).$

$\alpha_{12}=\alpha_5=0,$ then we have the algebra ${\bf TP}_6(g^2_{(5,1)}).$

\item If $\alpha_{14}=0$  then we derive $$\alpha_5'=\alpha_5+\frac{2A_2}{A_1}\alpha_9, \ \alpha_{9}'=\frac{B_2}{A_1}\alpha_9, \ \alpha_{12}'=\frac{B_2}{A_1}\alpha_{12}.$$

By using the associative identity for the triples $\{e_1,e_1,e_2\}$ and $\{e_1,e_2,e_2\}$, we obtain
    $$\begin{array}{l}
    2 \alpha _7 \alpha _9=\alpha _5 \alpha _{12}+2\alpha _9 \alpha _{15},\\[1mm]
    \alpha _5 \alpha _7-\alpha _7^2-\alpha _3 \alpha _9-2 \alpha _9 \alpha _{10}+\alpha _3 \alpha _{12}+2 \alpha _5 \alpha _{15}=0.\end{array}$$
Now we consider the following distinct cases:

$\alpha_{9}\neq0$, then by choosing $A_1=B_2\alpha_{9}, A_2=-\frac{B_2\alpha_5}{2}$, we can assume $\alpha_5'=0, \alpha_9'=1$ and obtain the algebra ${\bf TP}_7(g^2_{(5,1)}).$

$\alpha_9=0, \ \alpha_{12}\neq0,$ then by selecting  $A_1=B_2\alpha_{12}$,
we get $\alpha_{12}'=1$ and the algebra ${\bf TP}_8(g^2_{(5,1)}).$

$\alpha_9=0, \ \alpha_{12}=0, \alpha_5\neq0 $, then we have the algebra ${\bf TP}_9(g^2_{(5,1)}).$

$\alpha_9=\alpha_{12}=\alpha_5=0,$ then we have the algebra ${\bf TP}_{10}(g^2_{(5,1)}).$

\end{enumerate}\end{enumerate}\end{enumerate}\end{proof}

\begin{theorem} Let $(g^2_{(6,1)}, \cdot, [-,-])$ be a transposed Poisson algebra structure defined on the Lie algebra
$g^2_{(6,1)}$. Then the multiplication of $(g^2_{(6,1)}, \cdot)$ has the following form:

${\bf TP}_1(g^2_{(6,1)}):\ \begin{cases}
e_1\cdot e_1=e_2+ \alpha_{2}e_3+ \alpha_{3}e_4 +\alpha_{4}e_5, \ e_1\cdot e_2=-2\alpha_{11}e_4+\alpha_{7}e_5,\ e_1\cdot e_3=-\alpha_{11}e_5, \\[1mm]
e_1\cdot e_6=-e_3-\alpha_{2}e_4+\alpha_{8}e_5, \ e_2\cdot e_6=\alpha_{11}e_5, \ e_6\cdot e_6=e_4+\alpha_{12}e_5;
\end{cases}$

${\bf TP}_2(g^2_{(6,1)}):\ \begin{cases}
e_1\cdot e_1=e_3+ \alpha_{3}e_4+ \alpha_{4}e_5, \ e_1\cdot e_2=\alpha_{6}e_4+\alpha_{7}e_5,\ e_1\cdot e_3=\frac{1}{2}\alpha_{6}e_5, \\[1mm]
e_1\cdot e_6=-e_4+\alpha_{8}e_5, \ e_2\cdot e_2=\alpha_{10}e_5,\ e_2\cdot e_6=\alpha_{11}e_5, \ e_6\cdot e_6=\alpha_{12}e_5;
\end{cases}$

${\bf TP}_3(g^2_{(6,1)}):\ \begin{cases}
e_1\cdot e_1= \alpha_{3}e_4+\alpha_{4}e_5, \ e_1\cdot e_2=\alpha_{6}e_4+\alpha_{7}e_5,\ e_1\cdot e_3=\frac{1}{2}\alpha_{6}e_5, \ e_1\cdot e_6=\alpha_{8}e_5, \\[1mm]
e_2\cdot e_2=\alpha_{9}e_4+\alpha_{10}e_5,\ e_2\cdot e_3= \frac{1}{2}\alpha_{9}e_5, \ e_2\cdot e_6=\alpha_{11}e_5, \ e_6\cdot e_6=\alpha_{12}e_5;
\end{cases}$ \\
where the transposed Poisson algebra has its products with respect to the bracket $[-,-]$, and remaining products are equal to zero.
\end{theorem}

\begin{proof}
Let $(g^2_{(6,1)}, \cdot, [-,-])$ be a transposed Poisson algebra structure defined on Lie algebra $g^2_{(6,1)}$. Then,
for any element of $x \in g^2_{(6,1)} $, we have that operator of multiplication $\varphi_x(y) = x \cdot y$ is a $\frac12$-derivation. Hence, by using Theorem \ref{halfderiv2} for $1 \leq i\leq 6$,  we derive the following:

$$\begin{cases}
\varphi_{e_i}(e_1)=\alpha_{i,1}e_1+ \alpha_{i,2}e_2+ \alpha_{i,3}e_3+\alpha_{i,4}e_4+\alpha_{i,5}e_5+\alpha_{i,6}e_6, \\
\varphi_{e_i}(e_2)=\alpha_{i,1}e_2-3\alpha_{i,6}e_3+\beta_{i,4}e_4+\beta_{i,5}e_5,\\
\varphi_{e_i}(e_3)=\alpha_{i,1}e_3-2\alpha_{i,6}e_4 + \frac{1}{2}\beta_{i,4}e_5, \\
\varphi_{e_i}(e_4)=\alpha_{i,1}e_4-\frac{3}{2}\alpha_{i,6}e_5, \\
\varphi_{e_i}(e_5)=\alpha_{i,1}e_5, \\
\varphi_{e_i}(e_6)=-\alpha_{i,2}e_3-\alpha_{i,3}e_4+\gamma_{i,5}e_5+\alpha_{i,1}e_6, \\
  \end{cases}$$

By considering the equality  $\varphi_{e_i}(e_j)=e_i\cdot e_j=e_j\cdot e_i=\varphi_{e_j}(e_i)$ for the some elements ${e_i}$ and $e_j$, we have the following restrictions:

\begin{longtable}{llllllll}
   $ \{e_1, e_2\},$ & $\Rightarrow $ & $\alpha_{2,1}=0, \  \alpha_{2,2}=\alpha_{1,1}, \  \alpha_{2,3}=-3\alpha_{1,6}, \ \alpha_{2,4}=\beta_{1,4}, \ \alpha_{2,5}=\beta_{1,5}, \ \alpha_{2,6}=0,$\\
    $\{e_1, e_3\},$ & $\Rightarrow$ & $\alpha_{3,1}=\alpha_{3,2}=0, \  \alpha_{3,3}=\alpha_{1,1}, \  \alpha_{3,4}=-2\alpha_{1,6}, \ \alpha_{3,5}=\frac{1}{2}\beta_{1,4}, \ \alpha_{3,6}=0,$  \\
    $\{e_1, e_4\}, $& $\Rightarrow$ & $\alpha_{4,1}=\alpha_{4,2}=\alpha_{4,3}=0, \  \alpha_{4,4}=\alpha_{1,1}, \ \alpha_{4,5}=-\frac{3}{2}\alpha_{1,6}, \ \alpha_{4,6}=0,$  \\
    $\{e_1, e_5\},$ &$ \Rightarrow$ & $\alpha_{5,1}=\alpha_{5,2}=\alpha_{5,3}=\alpha_{5,4}=0, \ \alpha_{5,5}=\alpha_{1,1}, \ \alpha_{5,6}=0,$  \\
    $\{e_1, e_6\}, $&$ \Rightarrow $& $\alpha_{6,1}=\alpha_{6,2}=0, \
\alpha_{6,3}=-\alpha_{1,2}, \ \alpha_{6,4}=-\alpha_{1,3}, \ \alpha_{6,5}=\gamma_{1,5}, \ \alpha_{6,6}=\alpha_{1,1}, $ \\
 $\{e_2, e_3\},$ & $\Rightarrow$ & $\beta_{3,4}=0, \ \beta_{3,5}=\frac{1}{2}\beta_{2,4}, $ \\
 $\{e_2, e_4\},$ & $\Rightarrow$ & $\beta_{4,4}=\beta_{4,5}=0, $ \\
 $\{e_2, e_5\},$ & $\Rightarrow$ & $\beta_{5,4}=\beta_{5,5}=0,$\\
 $\{e_2, e_6\},$ &$ \Rightarrow$ & $\alpha_{1,1}=0, \ \beta_{6,4}=3\alpha_{1,6}, \  \beta_{6,5}=\gamma_{2,5},$\\
$\{e_3, e_6\},$ & $\Rightarrow$ & $\gamma_{3,5}=\frac{3}{2}\alpha_{1,6},$\\
$\{e_4, e_6\},$ & $\Rightarrow$ & $\gamma_{4,5}=0,$\\
$\{e_5, e_6\},$ & $\Rightarrow $& $\gamma_{5,5}=0.$
\end{longtable}

Thus, we have the following table of commutative multiplications of the transposed Poisson algebra structure defined on  $g^2_{(6,1)}$:

$$\begin{cases}
e_1\cdot e_1=\alpha_{1}e_2+ \alpha_{2}e_3+ \alpha_{3}e_4+\alpha_{4}e_5+\alpha_{5}e_6, \ e_1\cdot e_2=-3\alpha_{5}e_3+\alpha_{6}e_4+ \alpha_{7}e_5,\\[1mm]
e_1\cdot e_3=-2\alpha_{5}e_4+\frac{1}{2}\alpha_{6}e_5, \ e_1\cdot e_4=-\frac{3}{2}\alpha_{5}e_5, \ e_1\cdot e_6=-\alpha_{1}e_3-\alpha_{2}e_4+\alpha_{8}e_5, \\[1mm]
e_2\cdot e_2=\alpha_{9}e_4+\alpha_{10}e_5,\ e_2\cdot e_3= \frac{1}{2}\alpha_{9}e_5, \ e_2\cdot e_6=3\alpha_{5}e_4+\alpha_{11}e_5, \\[1mm]
e_3\cdot e_6=\frac32\alpha_{5}e_5, \ e_6\cdot e_6=\alpha_{1}e_4 +\alpha_{12}e_5.
\end{cases}$$

By using the multiplication of Lie algebra $g^2_{(6,1)}$, we consider the general basis change:
$$e_1'=\sum\limits_{t=1}^{6}A_te_t, \quad e_2'=\sum\limits_{t=1}^{6}B_te_t, \quad e_6'=\sum\limits_{t=1}^{6}C_te_t.$$
Then the equality $e_1'\cdot e_1'=\alpha_{1}'e_2'+ \alpha_{2}'e_3'+ \alpha_{3}'e_4'+\alpha_{4}'e_5'+\alpha_{5}'e_6'$ derives
$$\alpha_1'=\frac{A_1^2}{B_2}\alpha_1.$$

\begin{enumerate}
    \item If $\alpha_1\neq0.$ Then  by choosing $B_2=A_1^2\alpha_1$ we can assume $\alpha_1'=1$. Further, considering the associative identity  for the triples $\{e_1,e_1,e_6\}$ and $\{e_1,e_1,e_2\}$, we obtain the following restrictions, respectively:
$$\alpha_{5}=0, \ \alpha_{6}=-2\alpha _{11}, \ \alpha_{9}=\alpha_{10}=0.$$
In this case we obtain the algebra  ${\bf TP}_1(g^2_{(6,1)}).$

\item If $\alpha_1=0.$ Using the associative identity for the triple $\{e_1,e_1,e_3\}$ we obtain $\alpha_5=0$. Again, using the change of basis we get the expression
$$\alpha_2'=\frac{A_1}{B_2}\alpha_2.$$

\begin{enumerate}
    \item If $\alpha_2\neq0,$ then putting $B_2=A_1\alpha_2$ we get $\alpha_2'=1.$ Furthermore, from the identity $e_1\cdot(e_1\cdot e_3)=(e_1\cdot e_1)\cdot e_3$ we obtain $\alpha_9=0$. Hence, we have the algebra ${\bf TP}_2(g^2_{(6,1)}).$

\item If $\alpha_2=0,$ then we have the algebras ${\bf TP}_3(g^2_{(6,1)}).$
\end{enumerate}

\end{enumerate}

\end{proof}

\begin{theorem} Let $(g^2_{(n,1)}, \cdot, [-,-])$ be a transposed Poisson algebra structure defined on Lie algebra
$g^2_{(n,1)}$ and $n\geq 7$. Then, the multiplication of $(g^2_{(n,1)}, \cdot)$ has the following form:

${\bf TP}_1(g^2_{(n,1)}):\ \begin{cases}
e_1\cdot e_1=e_2+\sum\limits_{i=3}^{n-1}\alpha_{i}e_i, \ e_1\cdot e_2=2\alpha_{1}e_{n-2}+\alpha_{n}e_{n-1}, \ e_1\cdot e_3=\alpha_{1}e_{n-1}, \\[1mm]
e_1\cdot e_n=-e_3-\sum\limits_{i=4}^{n-2}\alpha_{i-1}e_i+\alpha_{n+1}e_{n-1},\ e_2\cdot e_n=-\alpha_1\alpha_{3}e_{n-1},\\[1mm] e_n\cdot e_n=\sum\limits_{i=4}^{n-2}\alpha_{i-1}e_i +\alpha_{n+2}e_{n-1};
\end{cases}$

${\bf TP}_2(g^2_{(n,1)}):\ \begin{cases}
e_1\cdot e_1=e_3+\sum\limits_{i=4}^{n-1}\alpha_{i}e_i, \ e_1\cdot e_2=\alpha_{n}e_{n-1}, \ e_1\cdot e_n=-e_4-\sum\limits_{i=5}^{n-2}\alpha_{i-1}e_i+\alpha_{n+1}e_{n-1},\\[1mm]
e_2\cdot e_2=\alpha_{n+3}e_{n-1}, \ e_2\cdot e_n=\alpha_{n+4}e_{n-1},\ e_n\cdot e_n=e_4+\sum\limits_{i=5}^{n-2}\alpha_{i-1}e_i +\alpha_{n+5}e_{n-1};
\end{cases}$

${\bf TP}_3(g^2_{(n,1)}):\ \begin{cases}
e_1\cdot e_1=\sum\limits_{i=4}^{n-1}\alpha_{i}e_i, \ e_1\cdot e_2=2\alpha_{1}e_{n-2}+\alpha_{n}e_{n-1}, \ e_1\cdot e_3=\alpha_{1}e_{n-1}, \\[1mm]
e_1\cdot e_n=-\sum\limits_{i=5}^{n-2}\alpha_{i-1}e_i+\alpha_{n+1}e_{n-1},\ e_2\cdot e_2=2\alpha_{n+2}e_{n-2}+\alpha_{n+3}e_{n-1}, \\[1mm]
e_2\cdot e_3=\alpha_{n+2}e_{n-1}, \ e_2\cdot e_n=\alpha_{n+4}e_{n-1},\ e_n\cdot e_n=\sum\limits_{i=5}^{n-2}\alpha_{i-1}e_i +\alpha_{n+5}e_{n-1};
\end{cases}$\\
where the transposed Poisson algebra has its products with respect to the bracket $[-,-]$, and the remaining products are equal to zero.
\end{theorem}

\begin{proof}
Let $(g^2_{(n,1)}, \cdot, [-,-])$ be a transposed Poisson algebra structure defined on Lie algebra $g^2_{(n,1)}$. Then,
for any element $x \in g^2_{(n,1)} $, we have operator of multiplication $\varphi_x(y) = x \cdot y$ is a $\frac12$-derivation. Hence, for $1 \leq i \leq n$, we derive

$$\begin{cases}
\varphi_{e_i}(e_1)=\sum\limits_{j=1}^{n-1}\alpha_{i,j}e_j, \\
\varphi_{e_i}(e_2)=\alpha_{i,1}e_2+\beta_{i,n-2}e_{n-2}+\beta_{i,n-1}e_{n-1}, \\
\varphi_{e_i}(e_3)=\alpha_{i,1}e_3+\frac{1}{2}\beta_{i,n-2}e_{n-1}, \\
\varphi_{e_i}(e_j)=\alpha_{i,1}e_j, \  4\leq j\leq n-1,\\
\varphi_{e_i}(e_n)=-\sum\limits_{j=3}^{n-2}\alpha_{i,j-1}e_j+\gamma_{i,n-1}e_{n-1}+\alpha_{i,1}e_n.
 \end{cases}$$

It is known that $\varphi_{e_i}(e_j)=e_i\cdot e_j=e_j\cdot e_i=\varphi_{e_j}(e_i).$ Then for some pair elements, we derive the following restrictions:

\begin{longtable}{llllll}
    $\{e_1, e_2\}, $ & $\Rightarrow$ &$\alpha_{2,1}=0, \  \alpha_{2,2}=\alpha_{1,1}, \  \alpha_{2,j}=0, 3\leq j\leq n-3, $ \\

    & &$\alpha_{2,n-2}=\beta_{1,n-2}, \ \alpha_{2,n-1}=\beta_{1,n-1}, $\\
    
    $\{e_1, e_3\},$ & $\Rightarrow$ & $\alpha_{3,1}=\alpha_{3,2}=0, \  \alpha_{3,3}=\alpha_{1,1}, \  \alpha_{3,j}=0, 4\leq j\leq n-2,$ \\
      & & $\alpha_{3,n-1}=\frac{1}{2}\beta_{1,n-2},$  \\
    $\{e_1, e_j\}, $ & $\Rightarrow$ &$\alpha_{j,j}=\alpha_{1,1},\  \alpha_{j,t}=0, \  1\leq t\neq j\leq n-1,$  \\
   $4\leq j\leq n-1, $ & & \\
    $\{e_1, e_n\}, $ & $\Rightarrow$ & $\alpha_{n,1}=\alpha_{n,2}=0, \  \alpha_{n,t}=-\alpha_{1,t-1}, \ 3\leq t\leq n-2, $ \\
     & & $\alpha_{n,n-1}=\gamma_{1,n-1},\  \alpha_{1,1}=0,$ \\

 $\{e_2, e_3\}, $ & $\Rightarrow$ & $\beta_{3,n-2}=0, \ \beta_{3,n-1}=\frac{1}{2}\beta_{2,n-2},$  \\
 $\{e_2, e_j\},$ & $\Rightarrow $ & $\beta_{j,n-2}=\beta_{j,n-1}=0,$ \\
 $4\leq j\leq n-1,$ & &\\
 $\{e_2, e_n\},$ & $\Rightarrow $ & $\beta_{n,n-2}=0, \  \beta_{n,n-1}=\gamma_{2,n-1},$\\
 $ \{e_i, e_n\},$ & $\Rightarrow $ & $\gamma_{i,n-1}=0, \quad 3\leq i\leq n-1.$
\end{longtable}

Thus, we have the following table of commutative multiplications of the transposed Poisson algebra structure defined on  $g^2_{(n,1)}$:

\begin{equation}\label{2klas}
\begin{cases}
e_1\cdot e_1=\sum\limits_{i=2}^{n-1}\alpha_{i}e_i, \ e_1\cdot e_2=2\alpha_{1}e_{n-2}+\alpha_{n}e_{n-1}, \ e_1\cdot e_3=\alpha_{1}e_{n-1}, \\
e_1\cdot e_n=-\sum\limits_{i=3}^{n-2}\alpha_{i-1}e_i+\alpha_{n+1}e_{n-1},\ e_2\cdot e_2=2\alpha_{n+2}e_{n-2}+\alpha_{n+3}e_{n-1}, \\
e_2\cdot e_3=\alpha_{n+2}e_{n-1}, \ e_2\cdot e_n=\alpha_{n+4}e_{n-1},\ e_n\cdot e_n=\sum\limits_{i=4}^{n-2}\alpha_{i-1}e_i +\alpha_{n+5}e_{n-1}.
\end{cases}\end{equation}

Considering the associative identity $x\cdot (y\cdot z)=(x\cdot y)\cdot z$ in \eqref{2klas}, we obtain the following restrictions on structure constants:
$$\begin{array}{lll}
\{e_1,e_1,e_2\}, & \Rightarrow & \alpha_{2}\alpha_{n+2}=0, \ \alpha_{2}\alpha_{n+3}+ \alpha_{3}\alpha_{n+2}=0,  \\[1mm]
\{e_1,e_1,e_n\}, & \Rightarrow & \alpha_{1}\alpha_{3}+\alpha_{2}\alpha_{n+4}=0.
\end{array}$$

Similarly, by using the multiplication of Lie algebra $g^2_{(n,1)}$ we consider the general basis change:
$$e_1'=\sum\limits_{t=1}^{n}A_te_t, \ e_2'=\sum\limits_{t=1}^{n}B_te_t, \ e_n'=\sum\limits_{t=1}^{n}C_te_t.$$
Then the product $e_1'\cdot e_1'=\sum\limits_{i=2}^{n-1}\alpha_{i}'e_i'$ gives
$$\alpha_2'=\frac{A_1^2}{B_2}\alpha_2.$$
We have the following cases.

\begin{enumerate}
    \item Let $\alpha_2\neq0.$ Then by choosing $B_2=A_1^2\alpha_2$, we put $\alpha_2'=1.$ In this case we obtain the algebra ${\bf TP}_1(g^2_{(n,1)}).$

    \item Let $\alpha_2=0$. Then from above restrictions we have $\alpha_3\alpha_{n+2}=\alpha_1\alpha_3=0.$ If we apply a general change of basis, then we have
    $$\alpha_3'=\frac{A_1}{B_2}\alpha_3.$$

    \begin{enumerate}
        \item If $\alpha_3\neq 0,$ then $\alpha_3'=1$ and we have the algebra ${\bf TP}_2(g^2_{(n,1)}).$
        \item If $\alpha_3=0,$ then we have the algebra ${\bf TP}_3(g^2_{(n,1)}).$
    \end{enumerate}

\end{enumerate}

\end{proof}

\begin{theorem}
Let $(g^3_{(7,1)}, \cdot, [-,-])$ be a transposed Poisson algebra structure defined on Lie algebra
$g^3_{(7,1)}$. Then, the multiplication of $(g^3_{(7,1)}, \cdot)$ has the following form:

${\bf TP}_1(g^3_{(7,1)}):\ \begin{cases}
e_1\cdot e_1=\alpha_{4}e_4+ \alpha_{5}e_5, \ e_1\cdot e_2=2\alpha_{6}e_5 +\alpha_{7}e_6,\ e_1\cdot e_3=\alpha_{6}e_6, \\[1mm]
e_1\cdot e_7=-\alpha_{4}e_5+ \alpha_{8}e_6, \ e_2\cdot e_2=2\alpha_{1}e_5+\alpha_{9}e_6,\ e_2\cdot e_3=\alpha_{1}e_6, \\[1mm]
e_2\cdot e_7=\alpha_{10}e_6,\ e_7\cdot e_7=\alpha_{11}e_6; \end{cases}$

${\bf TP}_2(g^3_{(7,1)}):\ \begin{cases}
e_1\cdot e_1=2e_3+\alpha_{4}e_4+\alpha_{5}e_5, \ e_1\cdot e_2=2\alpha_{6}e_5 +\alpha_{7}e_6,\ e_1\cdot e_3=(\alpha_{6}-1)e_6, \\[1mm]
e_1\cdot e_7=-2e_4-\alpha_{4}e_5+\alpha_{8}e_6, \ e_2\cdot e_2=\alpha_{9}e_6,\ e_2\cdot e_7=\alpha_{10}e_6,\\[1mm]
e_7\cdot e_7=2e_5+\alpha_{11}e_6;
\end{cases}$\\
where the transposed Poisson algebra has its products with respect to the bracket $[-,-]$, and remaining products are equal to zero.

\end{theorem}
\begin{proof} Let $(g^3_{(7,1)}, \cdot, [-,-])$ be a transposed Poisson algebra structure defined on Lie algebra $g^3_{(7,1)}$. Then
for any element $x \in g^3_{(7,1)} $, we have that operator of multiplication $\varphi_x(y) = x \cdot y$ is a $\frac12$-derivation. Hence, for $1 \leq i \leq 7$ by Theorem \ref{halfderiv3} we derive the following:

$$\begin{cases}
\varphi_{e_i}(e_1)=\alpha_{i,1}e_1+\alpha_{i,2}e_2+ \alpha_{i,3}e_3+\alpha_{i,4}e_4+\alpha_{i,5}e_5+\alpha_{i,6}e_6, \\
\varphi_{e_i}(e_2)=\alpha_{i,1}e_2+4\alpha_{i,2}e_4+ \beta_{i,5}e_5+\beta_{i,6}e_6,\\
\varphi_{e_i}(e_3)=\alpha_{i,1}e_3+2\alpha_{i,2}e_5+ \frac{1}{2}(\beta_{i,5}-\alpha_{i,3})e_6, \\
\varphi_{e_i}(e_4)=\alpha_{i,1}e_4+\frac{3}{2}\alpha_{i,2}e_6, \\
\varphi_{e_i}(e_5)=\alpha_{i,1}e_5, \\
\varphi_{e_i}(e_6)=\alpha_{i,1}e_6, \\
\varphi_{e_i}(e_7)=-\alpha_{i,2}e_3-\alpha_{i,3}e_4-\alpha_{i,4}e_5+\gamma_{i,6}e_6+\alpha_{i,1}e_7, \\
  \end{cases}$$

Considering the property $\varphi_{e_i}(e_j)=e_i\cdot e_j=e_j\cdot e_i=\varphi_{e_j}(e_i)$, we obtain the following restrictions:

 $$\begin{array}{lll}
    \{e_1, e_2\}, & \Rightarrow & \alpha_{2,1}=0, \  \alpha_{2,2}=\alpha_{1,1}, \  \alpha_{2,3}=0, \ \alpha_{2,4}=4\alpha_{1,2}, \ \alpha_{2,5}=\beta_{1,5}, \ \alpha_{2,6}=\beta_{1,6},\\[1mm]
    \{e_1, e_3\}, & \Rightarrow & \alpha_{3,1}=\alpha_{3,2}=0, \  \alpha_{3,3}=\alpha_{1,1}, \  \alpha_{3,4}=-2\alpha_{1,6}, \ \alpha_{3,5}=2\alpha_{1,2}, \ \alpha_{3,6}=\frac{1}{2}(\beta_{1,5}-\alpha_{1,3}),  \\[1mm]
    \{e_1, e_4\}, & \Rightarrow & \alpha_{4,1}=\alpha_{4,2}=\alpha_{4,3}=0, \  \alpha_{4,4}=\alpha_{1,1}, \  \alpha_{4,5}=0, \ \alpha_{4,6}=\frac{3}{2}\alpha_{1,2},  \\[1mm]
    \{e_1, e_5\}, & \Rightarrow & \alpha_{5,1}=\alpha_{5,2}=\alpha_{5,3}=\alpha_{5,4}=0, \ \alpha_{5,5}=\alpha_{1,1}, \ \alpha_{5,6}=0,  \\[1mm]
    \{e_1, e_6\}, & \Rightarrow & \alpha_{6,1}=\alpha_{6,2}=\alpha_{6,3}=\alpha_{6,4}=\alpha_{6,5}=0, \ \alpha_{6,6}=\alpha_{1,1},  \\[1mm]
   \{e_1, e_7\}, & \Rightarrow & \alpha_{1,1}=0, \ \alpha_{7,1}=\alpha_{7,2}=0, \ \alpha_{7,3}=-\alpha_{1,2}, \ \alpha_{7,4}=-\alpha_{1,3}, \ \alpha_{7,5}=-\alpha_{1,4}, \ \alpha_{7,6}=\gamma_{1,6},  \\[1mm]
\{e_2, e_3\}, & \Rightarrow & \beta_{3,5}=0, \ \beta_{3,6}=\frac{1}{2}\beta_{2,5},  \\[1mm]
 \{e_2, e_4\}, & \Rightarrow & \beta_{4,5}=\beta_{4,6}=0,  \\[1mm]
 \{e_2, e_5\}, & \Rightarrow & \beta_{5,5}=\beta_{5,6}=0,\\[1mm]
 \{e_2, e_6\}, & \Rightarrow & \beta_{6,5}=\beta_{6,6}=0,\\[1mm]
 \{e_2, e_7\}, & \Rightarrow & \beta_{7,5}=-4\alpha_{1,2}, \  \beta_{7,6}=\gamma_{2,6},\\[1mm]
\{e_3, e_7\}, & \Rightarrow & \alpha_{1,6}=0, \ \gamma_{3,6}=-\frac{3}{2}\alpha_{1,2},\\[1mm]
\{e_4, e_7\}, & \Rightarrow & \gamma_{4,6}=0,\\[1mm]
\{e_5, e_7\}, & \Rightarrow & \gamma_{5,6}=0,\\[1mm]
\{e_6, e_7\}, & \Rightarrow & \gamma_{6,6}=0.\\[1mm]
    \end{array}$$

Thus, we have the following table of commutative multiplications of the transposed Poisson algebra structure $(g^3_{(7,1)},, \cdot, [-,-])$:

$$\begin{cases}
e_1\cdot e_1=2\alpha_{2}e_2+2\alpha_{3}e_3+ \alpha_{4}e_4+\alpha_{5}e_5, \ e_1\cdot e_2=8\alpha_{2}e_4+2\alpha_{6}e_5 +\alpha_{7}e_6,\\[1mm]
e_1\cdot e_3=4\alpha_{2}e_5+(\alpha_{6}-\alpha_{3})e_6, \ e_1\cdot e_4=3\alpha_{2}e_6, \ e_1\cdot e_7=-2\alpha_{2}e_3-2\alpha_{3}e_4-\alpha_{4}e_5+\alpha_{8}e_6, \\[1mm]
e_2\cdot e_2=2\alpha_{1}e_5+\alpha_{9}e_6,\ e_2\cdot e_3=\alpha_{1}e_6, \ e_2\cdot e_7=-8\alpha_{2}e_5+\alpha_{10}e_6,\\[1mm]
e_3\cdot e_7=-3\alpha_{2}e_6, \ e_7\cdot e_7=2\alpha_{2}e_4+2\alpha_{3}e_5 +\alpha_{11}e_6.
\end{cases}$$

Considering the associative identity  for the triples $\{e_1,e_1,e_7\}$ and $\{e_1,e_1,e_2\}$, we obtain the following restrictions, respectively:
$$\alpha_2=0, \ \alpha_{1}\alpha_{3}=0.$$

Using  multiplication of Lie algebra $g^3_{(7,1)}$ and multiplication $e_1\cdot e_1=2\alpha_{3}e_3+ \alpha_{4}e_4+\alpha_{5}e_5$, and also consider a general change of basis, we can show that $\alpha_3$ is an invariant.
%
%

If $\alpha_{3}=0.$ Then, we have the algebra ${\bf TP}_1(g^3_{(7,1)}).$

If $\alpha_{3}\neq 0.$ Then, we obtain $\alpha_{1}=0$. By choosing change of basis, we can assume that $\alpha_3=1$ and obtain the algebra ${\bf TP}_2(g^3_{(7,1)}).$
\end{proof}

\begin{theorem}\label{thm3klas}
Let $(g^3_{(n,1)}, \cdot, [-,-])$ be a transposed Poisson algebra structure defined on Lie algebra $g^3_{(n,1)}$ and $n\geq 8$. Then, the multiplication of $(g^3_{(n,1)}, \cdot)$ has the following form:

${\bf TP}(g^3_{(8,1)}):\ \begin{cases}
e_1\cdot e_1=\sum\limits_{t=4}^{7} \alpha_{t}e_t, \ e_1\cdot e_2=\beta_{1}e_{6}+\beta_{2}e_{7}, \ e_1\cdot e_3=\frac{1}{2}(\beta_{1}-\alpha_{4})e_{7}, \\[1mm]
e_1\cdot e_8=-\alpha_{4}e_5-\alpha_{5}e_6 +\gamma_{1}e_{7},\ e_2\cdot e_2=\beta_{3}e_{6}+\beta_{4}e_{7}, \\[1mm]
e_2\cdot e_3=\frac{1}{2}\beta_{3}e_{7}, \ e_2\cdot e_8=\gamma_{2}e_{7},\ e_8\cdot e_8=\alpha_{4}e_6+\gamma_{3}e_{7};\end{cases}$

${\bf TP}_1(g^3_{(9,1)}):\ \begin{cases}
e_1\cdot e_1=\sum\limits_{t=4}^{8} \alpha_{t}e_t, \ e_1\cdot e_2=\beta_{1}e_{7}+\beta_{2}e_{8}, \ e_1\cdot e_3=\frac{1}{2}(\beta_{1}-\alpha_{5})e_{8}, \\[1mm]
e_1\cdot e_9=-\sum\limits_{t=5}^{7}\alpha_{t-1}e_t +\gamma_{1}e_{8},\ e_2\cdot e_2=\beta_{3}e_{7}+\beta_{4}e_{8}, \ e_2\cdot e_3=\frac{1}{2}\beta_{3}e_{8}, \\[1mm]
e_2\cdot e_9=-\alpha_{4}e_7+\gamma_{2}e_{8},\ e_3\cdot e_9=\frac12\alpha_{4}e_{8}, \ e_9\cdot e_9=\sum\limits_{t=6}^{7}\alpha_{t-2}e_t+\gamma_{3}e_{8};\end{cases}$

${\bf TP}_2(g^3_{(9,1)}):\ \begin{cases}
e_1\cdot e_1=e_3+\sum\limits_{t=5}^{8} \alpha_{t}e_t, \ e_1\cdot e_2=e_5+\beta_{1}e_{7}+ \beta_{2}e_{8}, \\[1mm]
e_1\cdot e_3=\frac{1}{2}(\beta_{1}-\alpha_{5})e_{8}, \ e_1\cdot e_9=-e_4-\sum\limits_{t=6}^{7}\alpha_{t-1}e_t +\gamma_{1}e_{8},\\[1mm]
e_2\cdot e_2=e_{7}+\beta_{4}e_{8}, \ e_2\cdot e_9=-e_6+\gamma_{2}e_{8},\ e_9\cdot e_9=e_5+\alpha_{5}e_7 +\gamma_{3}e_{8};
\end{cases}$

${\bf TP}_1(g^3_{(n,1)}):\ \begin{cases}
e_1\cdot e_1=\sum\limits_{t=4}^{n-1} \alpha_{t}e_t, \ e_1\cdot e_2=\sum\limits_{t=6}^{n-3}\alpha_{t-2}e_t +\beta_{1}e_{n-2}+\beta_{2}e_{n-1}, \\[1mm]
e_1\cdot e_3=\frac{1}{2}(\beta_{1}-\alpha_{n-4})e_{n-1}, \ e_1\cdot e_n=-\sum\limits_{t=5}^{n-2}\alpha_{t-1}e_t +\gamma_{1}e_{n-1},\\[1mm]
e_2\cdot e_2=\sum\limits_{t=8}^{n-3}\alpha_{t-4}e_t +\beta_{3}e_{n-2}+\beta_{4}e_{n-1}, \ e_2\cdot e_3=\frac{1}{2}(\beta_{3}-\alpha_{n-6})e_{n-1}, \\[1mm]
e_2\cdot e_n=-\sum\limits_{t=7}^{n-2}\alpha_{t-3}e_t +\gamma_{2}e_{n-1},\ e_3\cdot e_n=\frac12\alpha_{n-5}e_{n-1}, \
e_n\cdot e_n=\sum\limits_{t=6}^{n-2}\alpha_{t-2}e_t+\gamma_{3}e_{n-1}; \end{cases}$

${\bf TP}_2(g^3_{(n,1)}):\ \begin{cases}
e_1\cdot e_1=e_3+\sum\limits_{t=4}^{n-6} \alpha_{t}e_t+\sum\limits_{t=n-4}^{n-1} \alpha_{t}e_t, \ e_1\cdot e_2=e_5+\sum\limits_{t=5}^{n-4}\alpha_{t-2}e_t +\beta_{1}e_{n-2}+\beta_{2}e_{n-1}, \\[1mm]
e_1\cdot e_3=\frac{1}{2}(\beta_{1}-\alpha_{n-4})e_{n-1}, \ e_1\cdot e_n=-e_4-\sum\limits_{t=5}^{n-5}\alpha_{t-1}e_t-\sum\limits_{t=n-3}^{n-2}\alpha_{t-1}e_t +\gamma_{1}e_{n-1},\\[1mm]
e_2\cdot e_2=e_7+\sum\limits_{t=8}^{n-2}\alpha_{t-4}e_t+\beta_{4}e_{n-1}, \ e_2\cdot e_n=-e_6-\sum\limits_{t=7}^{n-3}\alpha_{t-3}e_t +\gamma_{2}e_{n-1},\\[1mm]
e_n\cdot e_n=e_5+\sum\limits_{t=6}^{n-4}\alpha_{t-2}e_t+\alpha_{n-4}e_{n-2}+\gamma_{3}e_{n-1};\end{cases}$\\
where the transposed Poisson algebra has its products with respect to the bracket $[-,-]$, and remaining products are equal to zero.
\end{theorem}

\begin{proof}
Let $(g^3_{(n,1)}, \cdot, [-,-])$ be a transposed Poisson algebra structure defined on Lie algebra $g^3_{(n,1)}$. Then
for any element $x \in g^3_{(n,1)} $ we have that, operator of multiplication $\varphi_x(y) = x \cdot y$ is a $\frac12$ -derivation. Hence, for $1 \leq i \leq n$ we derive

$$\begin{cases}
\varphi_{e_i}(e_1)=\alpha_{i,1}e_1+\sum\limits_{t=3}^{n-1}\alpha_{i,t}e_t, \\
\varphi_{e_i}(e_2)=\alpha_{i,1}e_2+\sum\limits_{t=5}^{n-3}\alpha_{i,t-2}e_t+\beta_{i,n-2}e_{n-2}+\beta_{i,n-1}e_{n-1}, \\
\varphi_{e_i}(e_3)=\alpha_{i,1}e_3+\frac{1}{2}(\beta_{i,n-2}-\alpha_{i,n-4})e_{n-1}, \\
\varphi_{e_i}(e_j)=\alpha_{i,1}e_j, \  4\leq j\leq n-1,\\
\varphi_{e_i}(e_n)=-\sum\limits_{t=4}^{n-2}\alpha_{i,t-1}e_t+\gamma_{i,n-1}e_{n-1}+\alpha_{i,1}e_n.
 \end{cases}$$

It is known that $\varphi_{e_i}(e_j)=e_i\cdot e_j=e_j\cdot e_i=\varphi_{e_j}(e_i).$ Then for some pair elements, we derive the following restrictions:

\begin{longtable}{lllll}
    $\{e_1, e_2\},$ & $\Rightarrow$ &$\alpha_{1,1}=\alpha_{2,1}=\alpha_{2,3}=\alpha_{2,4}=0, \ \alpha_{2,t}=\alpha_{1,t-2},\  5\leq t\leq n-3, $\\
    & &$\alpha_{2,n-2}=\beta_{1,n-2}, \ \alpha_{2,n-1}=\beta_{1,n-1},$ \\
    $\{e_1, e_3\},$ & $\Rightarrow$ & $\alpha_{3,1}=0, \ \alpha_{3,t}=0, \ 3\leq t\leq n-2,\   \alpha_{3,n-1}=\frac{1}{2}(\beta_{1,n-2}-\alpha_{1,n-4}), $\\
    $ \{e_1, e_j\}, $ & $\Rightarrow $& $\alpha_{j,1}=0, \ \alpha_{j,t}=0, \  3\leq t\leq n-1, $ \\
   $4\leq j\leq n-1,$ & & \\
   $ \{e_1, e_n\},$ & $\Rightarrow$ & $\alpha_{n,1}=\alpha_{n,3}=0, \  \alpha_{n,t}=-\alpha_{1,t-1}, \ 4\leq t\leq n-2, \  \alpha_{n,n-1}=\gamma_{1,n-1}, $ \\
 $\{e_2, e_3\},$ & $\Rightarrow$ & $\beta_{3,n-2}=0, \ \beta_{3,n-1}=\frac{1}{2}(\beta_{2,n-2}-\alpha_{1,n-6}), $ \\
 $\{e_2, e_j\},$ & $\Rightarrow$ & $\beta_{j,n-2}=\beta_{j,n-1}=0,$ \\
 $4\leq j\leq n-1,$& &\\
 $\{e_2, e_n\},$ & $\Rightarrow$ &$ \beta_{n,n-2}=-\alpha_{1,n-5}, \  \beta_{n,n-1}=\gamma_{2,n-1},$\\
 $ \{e_3, e_n\},$ & $\Rightarrow$ & $\gamma_{3,n-1}=\frac{1}{2}\alpha_{1,n-5}, $ \\
$   \{e_i, e_n\},$ & $\Rightarrow $& $\gamma_{i,n-1}=0, $\\
 $4\leq i\leq n-1.$
\end{longtable}

So, we have the following table of commutative multiplications of the transposed Poisson algebra structure $(g^3_{(n,1)},, \cdot, [-,-])$

For $n=8:$

$$\begin{cases}
e_1\cdot e_1=\sum\limits_{t=3}^{7} \alpha_{t}e_t, \ e_1\cdot e_2=\alpha_{3}e_5+ \beta_{1}e_{6}+\beta_{2}e_{7}, \ e_1\cdot e_3=\frac{1}{2}(\beta_{1}-\alpha_{4})e_{7}, \\[1mm]
e_1\cdot e_8=-\sum\limits_{t=4}^{6}\alpha_{t-1}e_t +\gamma_{1}e_{7},\ e_2\cdot e_2=\beta_{3}e_{6}+\beta_{4}e_{7}, \ e_2\cdot e_3=\frac{1}{2}\beta_{3}e_{7}, \\[1mm]
e_2\cdot e_8=-\alpha_{3}e_6 +\gamma_{2}e_{7},\ e_3\cdot e_8=\frac12\alpha_{3}e_{7},\ e_8\cdot e_8=\sum\limits_{t=5}^{6}\alpha_{t-2}e_t+\gamma_{3}e_{7};
\end{cases} $$

For $n\geq9:$
$$\begin{cases}
e_1\cdot e_1=\sum\limits_{t=3}^{n-1} \alpha_{t}e_t, \ e_1\cdot e_2=\sum\limits_{t=5}^{n-3}\alpha_{t-2}e_t +\beta_{1}e_{n-2}+\beta_{2}e_{n-1}, \ e_1\cdot e_3=\frac{1}{2}(\beta_{1}-\alpha_{n-4})e_{n-1}, \\[1mm]
e_1\cdot e_n=-\sum\limits_{t=4}^{n-2}\alpha_{t-1}e_t +\gamma_{1}e_{n-1},\ e_2\cdot e_2=\sum\limits_{t=7}^{n-3}\alpha_{t-4}e_t +\beta_{3}e_{n-2}+\beta_{4}e_{n-1}, \\[1mm]
e_2\cdot e_3=\frac{1}{2}(\beta_{3}-\alpha_{n-6})e_{n-1}, \ e_2\cdot e_n=-\sum\limits_{t=6}^{n-2}\alpha_{t-3}e_t +\gamma_{2}e_{n-1},\\
e_3\cdot e_n=\frac12\alpha_{n-5}e_{n-1}, \ e_n\cdot e_n=\sum\limits_{t=5}^{n-2}\alpha_{t-2}e_t+\gamma_{3}e_{n-1}.\\
\end{cases}
$$

Considering the associative identity  for the triples $\{e_1,e_1,e_2\}$ and $\{e_1,e_1,e_n\}$, we obtain the following restrictions, respectively:

\begin{equation}\label{gn3}\alpha_{3}(\beta_{3}-\alpha_{n-6})=0, \ \alpha_{3}\alpha_{n-5}=0.\end{equation}

We have the following cases.

\begin{enumerate}
    \item Let $n=8.$ Then we get $\alpha_3=0$ and we have the algebra ${\bf TP}(g^3_{(8,1)})$.

    \item Let $n\geq 9$. Then by using the multiplication of the Lie algebra $g^3_{(n,1)}$ we consider the general basis change:
$$e_1'=\sum\limits_{t=1}^{n}A_te_t, \ e_2'=\sum\limits_{t=1}^{n}B_te_t, \ e_n'=\sum\limits_{t=1}^{n}C_te_t.$$
The product $e_1'\cdot e_1'=\sum\limits_{i=3}^{n-1}\alpha_{i}'e_i'$ gives
$$\alpha_3'=\frac{A_1}{B_2}\alpha_3.$$
\begin{enumerate}
        \item If $\alpha_{3}=0.$ Then for $n=9$ we derive the algebra ${\bf TP}_1(g^3_{(9,1)})$ and for the case when $n\geq 10$ we obtain ${\bf TP}_1(g^3_{(n,1)})$.
        \item If $\alpha_{3}\neq 0.$ Then  by choosing $B_2=A_1\alpha_3$ we can assume $\alpha_{3}'=1$ and from the restrictions \eqref{gn3} we get $\alpha_{n-5}=0, \beta_3'=\alpha_{n-6}$. In this case for the $n=9$ we obtain ${\bf TP}_2(g^3_{(9,1)})$ and for $n\geq 10$ we have the algebra ${\bf TP}_2(g^3_{(n,1)})$.
    \end{enumerate}
\end{enumerate}

\end{proof}

\begin{theorem}\label{thm4klas} Let $(g^1_{7}, \cdot, [-,-])$ be a transposed Poisson algebra structure defined on Lie algebra
$g^1_{7}$. Then, the multiplication of $(g^1_{7}, \cdot)$ has the following form:

${\bf TP}_1(g^1_{7}): \ \begin{cases}
e_1\cdot e_1=\alpha_{4}e_4 +\alpha_{5}e_5 +\alpha_{6}e_6+\alpha_{7}e_7, \ e_1\cdot e_2=\alpha_{1}e_5+\alpha_{2}e_6+\alpha_{8}e_7,\\[1mm]
e_1\cdot e_3=\frac{1}{2}(\alpha_{1}-\alpha_{4})e_6-\frac{1}{2}\alpha_{5}e_7, \ e_1\cdot e_4=\frac{1}{2}\alpha_{4}e_7, \\[1mm]
e_2\cdot e_2=\alpha_{9}e_5+\alpha_{10}e_6+ \alpha_{11}e_7,\ e_2\cdot e_3=\frac{1}{2}\alpha_{9}e_6-\frac{1}{2}\alpha_{1}e_7;
\end{cases}$

${\bf TP}_2(g^1_{7}): \ \begin{cases}
e_1\cdot e_1=6e_3+\alpha_{4}e_4 +\alpha_{5}e_5+\alpha_{6}e_6+\alpha_{7}e_7, \ e_1\cdot e_2=-2e_4+ \alpha_{1}e_5+ \alpha_{2}e_6+\alpha_{8}e_7,\\[1mm]
e_1\cdot e_3=-4e_5+\frac{1}{2}(\alpha_{1}-\alpha_{4})e_6-\frac{1}{2}\alpha_{5}e_7, \ e_1\cdot e_4=-2e_6+ \frac{1}{2}\alpha_{4}e_7, \ e_1\cdot e_5=-3e_7, \\[1mm]
e_2\cdot e_2=-\frac23e_5+\alpha_{10}e_6+ \alpha_{11}e_7,\ e_2\cdot e_3=\frac{2}{3}e_6-\frac{1}{2}\alpha_{1}e_7, \  e_2\cdot e_4=-e_7, \  e_3\cdot e_3=2e_7;
\end{cases}$\\
where the transposed Poisson algebra has its products with respect to the bracket $[-,-]$, and remaining products are equal to zero.
\end{theorem}

\begin{proof}
Let $(g^1_{7}, \cdot, [-,-])$ be a transposed Poisson algebra structure defined on the Lie algebra $g^3_{(n,1)}$. Then
for any element of $x \in g^1_{7} $, we have that operator of multiplication $\varphi_x(y) = x \cdot y$ is a $\frac12$-derivation. Hence, for $1 \leq i\leq 7$ by using Theorem \ref{halfderiv4} we derive

 $$\begin{cases}
\varphi_{e_i}(e_1)=\alpha_{i,1}e_1+\alpha_{i,3}e_3+ \alpha_{i,4}e_4+\alpha_{i,5}e_5+\alpha_{i,6}e_6+\alpha_{i,7}e_7, \\
\varphi_{e_i}(e_2)=\alpha_{i,1}e_2-\frac{1}{3}\alpha_{i,3}e_4+ \beta_{i,5}e_5+\beta_{i,6}e_6+\beta_{i,7}e_7,\\
\varphi_{e_i}(e_3)=\alpha_{i,1}e_3-\frac{2}{3}\alpha_{i,3}e_5+ \frac{1}{2}(\beta_{i,5}-\alpha_{i,4})e_6-\frac{1}{2}\alpha_{i,5}e_7, \\
\varphi_{e_i}(e_4)=\alpha_{i,1}e_4-\frac{1}{3}\alpha_{i,3}e_6+ \frac{1}{2}\alpha_{i,4}e_7, \ \varphi_{e_i}(e_5)=\alpha_{i,1}e_5-\frac{1}{2}\alpha_{i,3}e_7, \\ \varphi_{e_i}(e_6)=\alpha_{i,1}e_6, \ \varphi_{e_i}(e_7)=\alpha_{i,1}e_7.
  \end{cases}$$

By considering equality  $\varphi_{e_i}(e_j)=e_i\cdot e_j=e_j\cdot e_i=\varphi_{e_j}(e_i)$ for the some elements ${e_i}$ and $e_j$, we have the following restrictions:

 $$\begin{array}{lll}
    \{e_1, e_2\}, & \Rightarrow & \alpha_{1,1}=\alpha_{2,1}=\alpha_{2,3}=0, \ \alpha_{2,4}=-\frac{1}{3}\alpha_{1,3}, \ \alpha_{2,5}=\beta_{1,5}, \ \alpha_{2,6}=\beta_{1,6}, \ \alpha_{2,7}=\beta_{1,7},\\[1mm]
    \{e_1, e_3\}, & \Rightarrow & \alpha_{3,1}=\alpha_{3,3}=\alpha_{3,4}=0, \  \alpha_{3,5}=-\frac{2}{3}\alpha_{1,3}, \ \alpha_{3,6}=\frac{1}{2}(\beta_{1,5}-\alpha_{1,4}), \ \alpha_{3,7}=-\frac{1}{2}\alpha_{1,5},  \\[1mm]
    \{e_1, e_4\}, & \Rightarrow & \alpha_{4,1}=\alpha_{4,3}=\alpha_{4,4}=\alpha_{4,5}=0, \  \alpha_{4,6}=-\frac{1}{3}\alpha_{1,3}, \ \alpha_{4,7}=\frac{1}{2}\alpha_{1,4},  \\[1mm]
    \{e_1, e_5\}, & \Rightarrow & \alpha_{5,1}=\alpha_{5,3}=\alpha_{5,4}=\alpha_{5,5}=\alpha_{5,6}=0, \ \alpha_{5,7}=-\frac{1}{2}\alpha_{1,3},  \\[1mm]
    \{e_1, e_6\}, & \Rightarrow & \alpha_{6,1}=\alpha_{6,3}=\alpha_{6,4}=\alpha_{6,5}=\alpha_{6,6}=\alpha_{6,7}=0,  \\[1mm]
   \{e_1, e_7\}, & \Rightarrow & \alpha_{7,1}=\alpha_{7,3}=\alpha_{7,4}=\alpha_{7,5}=\alpha_{7,6}=\alpha_{7,7}=0,  \\[1mm]
\{e_2, e_3\}, & \Rightarrow & \beta_{3,5}=0, \ \beta_{3,6}=\frac{1}{6}(3\beta_{2,5}+\alpha_{1,3}), \ \beta_{3,7}=-\frac{1}{2}\beta_{1,5},  \\[1mm]
 \{e_2, e_4\}, & \Rightarrow & \beta_{4,5}=\beta_{4,6}=0, \ \beta_{4,7}=-\frac{1}{6}\alpha_{1,3},  \\[1mm]
 \{e_2, e_5\}, & \Rightarrow & \beta_{5,5}=\beta_{5,6}=\beta_{5,7}=0,\\[1mm]
 \{e_2, e_6\}, & \Rightarrow & \beta_{6,5}=\beta_{6,6}=\beta_{6,7}=0,\\[1mm]
 \{e_2, e_7\}, & \Rightarrow & \beta_{7,5}=\beta_{7,6}=\beta_{7,7}=0.    \end{array}$$

So, we have the following table of commutative multiplications of the transposed Poisson algebra structure $(g^1_{(7)},, \cdot, [-,-])$

$$\begin{cases}
e_1\cdot e_1=\alpha_{3}e_3+\alpha_{4}e_4 +\alpha_{5}e_5+\alpha_{6}e_6+\alpha_{7}e_7, \ e_1\cdot e_2=-\frac{1}{3}\alpha_{3}e_4+ \alpha_{1}e_5+\alpha_{2}e_6+\alpha_{8}e_7,\\
e_1\cdot e_3=-\frac{2}{3}\alpha_{3}e_5+ \frac{1}{2}(\alpha_{1}-\alpha_{4})e_6-\frac{1}{2}\alpha_{5}e_7, \ e_1\cdot e_4=-\frac{1}{3}\alpha_{3}e_6+ \frac{1}{2}\alpha_{4}e_7, \ e_1\cdot e_5=-\frac{1}{2}\alpha_{3}e_7, \\
e_2\cdot e_2=\alpha_{9}e_5+\alpha_{10}e_6+ \alpha_{11}e_7,\ e_2\cdot e_3=\frac{1}{6}(3\alpha_{9}+\alpha_{3})e_6-\frac{1}{2}\alpha_{1}e_7, \\
e_2\cdot e_4=-\frac{1}{6}\alpha_{3}e_7, \ e_3\cdot e_3=\frac{1}{3}\alpha_{3}e_7, \\
\end{cases}$$

The equality $e_1\cdot (e_1\cdot e_2)=(e_1\cdot e_1)\cdot e_2$ gives $\alpha _3 \left(\alpha _3+9 \alpha _9\right)=0.$ Similarly above by using a general change of basis, we can show that $\alpha_3$ is an invariant.

\begin{enumerate}
    \item If $\alpha_3=0,$ then we have the algebra ${\bf TP}_1(g^1_{7})$;
    \item If $\alpha_3\neq0,$ then we get the algebra ${\bf TP}_2(g^1_{7})$.
\end{enumerate}
\end{proof}

\begin{theorem}\label{thm4klas} Let $(g^2_{9}, \cdot, [-,-])$ be a transposed Poisson algebra structure defined on the Lie algebra
$g^2_{9}$. Then the multiplication of $(g^2_{9}, \cdot)$ has the following form:
$${\bf TP}(g^2_{9}):\ \begin{cases}
e_1\cdot e_1=\alpha_{1}e_5+ \alpha_{2}e_6+ \alpha_{3}e_7+\alpha_{4}e_8+\alpha_{5}e_9, \ e_1\cdot e_2=\frac{1}{3}\alpha_{1}e_6+\alpha_{6}e_7+ \alpha_{7}e_8+\alpha_{8}e_9,\\[1mm]
e_1\cdot e_3=-\frac{4}{3}\alpha_{1}e_7+ \frac{1}{2}(\alpha_{6}- 5\alpha_{2})e_8-\frac{1}{2}\alpha_{3}e_9, \ e_1\cdot e_4=\frac{1}{3}\alpha_{1}e_8+ \frac{1}{2}\alpha_{2}e_9, \\[1mm]
e_1\cdot e_5=-\frac{1}{2}\alpha_{1}e_9, \ e_2\cdot e_2=\alpha_{9}e_7+\alpha_{10}e_8+ \alpha_{11}e_9,\\[1mm]
e_2\cdot e_3=\frac{1}{6}(3\alpha_{9}- 5\alpha_{1})e_8-\frac{1}{2}\alpha_{6}e_9, \ e_2\cdot e_4=\frac{1}{6}\alpha_{1}e_9, \ e_3\cdot e_3=\frac{2}{3}\alpha_{1}e_9,\end{cases}$$
where the transposed Poisson algebra has its products with respect to the bracket $[-,-]$, and remaining products are equal to zero.
\end{theorem}

\begin{proof}
Let $(g^2_{9}, \cdot, [-,-])$ be a transposed Poisson algebra structure defined on Lie algebra $g^2_{9}$. Then
for any element $x \in g^2_{9} $, we have that operator of multiplication $\varphi_x(y) = x \cdot y$ is a $\frac12$-derivation. Hence, for $1 \leq i\leq 9$ we derive

$$\begin{cases}
\varphi_{e_i}(e_1)=\alpha_{i,1}e_1+\alpha_{i,5}e_5+ \alpha_{i,6}e_6+\alpha_{i,7}e_7+\alpha_{i,8}e_8+\alpha_{i,9}e_9, \\
\varphi_{e_i}(e_2)=\alpha_{i,1}e_2+\frac{1}{3}\alpha_{i,5}e_6+ \beta_{i,7}e_7+ \beta_{i,8}e_8+\beta_{i,9}e_9,\\
\varphi_{e_i}(e_3)=\alpha_{i,1}e_3-\frac{4}{3}\alpha_{i,5}e_7+ \frac{1}{2}(\beta_{i,7}- 5\alpha_{i,6})e_8-\frac{1}{2}\alpha_{i,7}e_9, \\
\varphi_{e_i}(e_4)=\alpha_{i,1}e_4+\frac{1}{3}\alpha_{i,5}e_8+ \frac{1}{2}\alpha_{i,6}e_9, \ \varphi_{e_i}(e_5)=\alpha_{i,1}e_5-\frac{1}{2}\alpha_{i,5}e_9, \ \varphi_{e_i}(e_j)=\alpha_{i,1}e_j, \ \ 6\leq j\leq 9. \\
\end{cases}$$

It is known that $\varphi_{e_i}(e_j)=e_i\cdot e_j=e_j\cdot e_i=\varphi_{e_j}(e_i).$ Then for some pair elements, we derive the following restrictions:

$$\begin{array}{lll}
    \{e_1, e_2\}, & \Rightarrow & \alpha_{1,1}=\alpha_{2,1}=\alpha_{2,5}=0, \ \alpha_{2,6}=\frac{1}{3}\alpha_{1,5}, \ \alpha_{2,7}=\beta_{1,7}, \ \alpha_{2,8}=\beta_{1,8}, \ \alpha_{2,9}=\beta_{1,9},\\[1mm]
    \{e_1, e_3\}, & \Rightarrow & \alpha_{3,1}=\alpha_{3,5}=\alpha_{3,6}=0, \  \alpha_{3,7}=-\frac{4}{3}\alpha_{1,5}, \ \alpha_{3,8}=\frac{1}{2}(\beta_{1,7}- 5\alpha_{1,6}), \ \alpha_{3,9}=-\frac{1}{2}\alpha_{1,7},  \\[1mm]
    \{e_1, e_4\}, & \Rightarrow & \alpha_{4,1}=\alpha_{4,5}=\alpha_{4,6}=\alpha_{4,7}=0, \  \alpha_{4,8}=\frac{1}{3}\alpha_{1,5}, \ \alpha_{4,9}=\frac{1}{2}\alpha_{1,6},  \\[1mm]
    \{e_1, e_5\}, & \Rightarrow & \alpha_{5,1}=\alpha_{5,5}=\alpha_{5,6}=\alpha_{5,7}=\alpha_{5,8}=0, \ \alpha_{5,9}=-\frac{1}{2}\alpha_{1,5},  \\[1mm]
    \{e_1, e_j\}, & \Rightarrow & \alpha_{j,1}=\alpha_{j,5}=\alpha_{j,6}=\alpha_{j,7}=\alpha_{j,8}=\alpha_{j,9}, \ 6\leq j\leq 9, \\[1mm]
 \{e_2, e_3\}, & \Rightarrow & \beta_{3,7}=0, \ \beta_{3,8}=\frac{1}{6}(3\beta_{2,7}- 5\alpha_{1,5}), \ \beta_{3,9}=-\frac{1}{2}\beta_{1,7},  \\[1mm]
 \{e_2, e_4\}, & \Rightarrow & \beta_{4,7}=\beta_{4,8}=0, \ \beta_{4,9}=\frac{1}{6}\alpha_{1,3},  \\[1mm]
 \{e_2, e_j\}, & \Rightarrow & \beta_{j,7}=\beta_{j,8}=\beta_{j,9}=0, \ 5\leq j\leq 9.\\[1mm]
   \end{array}$$

So, we have the following table of multiplications of the transposed Poisson algebra structure $(g^2_{(9)},, \cdot, [-,-])$. Note the associative identity for any triple is hold.

$$\begin{cases}
e_1\cdot e_1=\alpha_{1}e_5+ \alpha_{2}e_6+ \alpha_{3}e_7+\alpha_{4}e_8+\alpha_{5}e_9, \ e_1\cdot e_2=\frac{1}{3}\alpha_{1}e_6+\alpha_{6}e_7+ \alpha_{7}e_8+\alpha_{8}e_9,\\
e_1\cdot e_3=-\frac{4}{3}\alpha_{1}e_7+ \frac{1}{2}(\alpha_{6}- 5\alpha_{2})e_8-\frac{1}{2}\alpha_{3}e_9, \ e_1\cdot e_4=\frac{1}{3}\alpha_{1}e_8+ \frac{1}{2}\alpha_{2}e_9, \\
e_1\cdot e_5=-\frac{1}{2}\alpha_{1}e_9, \ e_2\cdot e_2=\alpha_{9}e_7+\alpha_{10}e_8+ \alpha_{11}e_9,\\
e_2\cdot e_3=\frac{1}{6}(3\alpha_{9}- 5\alpha_{1})e_8-\frac{1}{2}\alpha_{6}e_9, \ e_2\cdot e_4=\frac{1}{6}\alpha_{1}e_9, \ e_3\cdot e_3=\frac{2}{3}\alpha_{1}e_9.
\end{cases}$$
\end{proof}
\begin{theorem}\label{thm5klas} Let $(g^3_{11}, \cdot, [-,-])$ be a transposed Poisson algebra structure defined on the Lie algebra
$g^3_{11}$. Then the multiplication of $(g^3_{11}, \cdot)$ has the following form:
$${\bf TP}(g^3_{11}):\ \begin{cases}
e_1\cdot e_1=\alpha_{6}e_6+ \alpha_{7}e_7+ \alpha_{8}e_8+ \alpha_{9}e_9+ \alpha_{10}e_{10}+ \alpha_{11}e_{11}, \\
e_1\cdot e_2=-\alpha_{6}e_7-\alpha_{7}e_8+\alpha_{1}e_9+ \alpha_{2}e_{10}+\alpha_{3}e_{11},\ e_1\cdot e_3=\frac{1}{2}\alpha_{1}e_{10}-\frac{1}{2}\alpha_{9}e_{11}, \\ e_1\cdot e_4=\frac{1}{2}\alpha_{7}e_{10}+ \frac{1}{2}\alpha_{8}e_{11}, \ e_1\cdot e_5=-\frac{1}{2}\alpha_{6}e_{10}-\frac{1}{2}\alpha_{7}e_{11}, \ e_1\cdot e_6=\frac{1}{2}\alpha_{6}e_{11}, \\
e_2\cdot e_2=\alpha_{6}e_8+\alpha_{4}e_9+ \alpha_{5}e_{10}+\alpha_{12}e_{11},\ e_2\cdot e_3=\frac{1}{2}\alpha_{4}e_{10}-\frac{1}{2}\alpha_{1}e_{11}, \\
e_2\cdot e_4=-\frac{1}{2}\alpha_{6}e_{10}- \frac{1}{2}\alpha_{7}e_{11}, \ e_2\cdot e_5=\frac{1}{2}\alpha_{6}e_{11},
\end{cases}$$
where it is taken into account that the transposed Poisson algebra has its products with respect to the bracket $[-,-]$, and the remaining products are equal to zero.
\end{theorem}

\begin{proof}
Let $(g^3_{11}, \cdot, [-,-])$ be a transposed Poisson algebra structure defined on Lie algebra $g^3_{11}$. Then
for any element of $x \in g^3_{11} $ we have that, operator of multiplication $\varphi_x(y) = x \cdot y$ is a $\frac12$-derivation. Further, using Theorem \ref{2halfderiv3} for $1 \leq i\leq 11$ we derive
$$\begin{cases}
\varphi_{e_i}(e_1)=\alpha_{i,1}e_1+\alpha_{i,6}e_6+ \alpha_{i,7}e_7+ \alpha_{i,8}e_8+ \alpha_{i,9}e_9+ \alpha_{i,10}e_{10}+\alpha_{i,11}e_{11}, \\

\varphi_{e_i}(e_2)=\alpha_{i,1}e_2-\alpha_{i,6}e_7-\alpha_{i,7}e_8+ \beta_{i,9}e_9+ \beta_{i,10}e_{10}+ \beta_{i,11}e_{11},\\

\varphi_{e_i}(e_3)=\alpha_{i,1}e_3+\frac{1}{2}\beta_{i,9}e_{10}-\frac{1}{2}\alpha_{i,9}e_{11}, \\

\varphi_{e_i}(e_4)=\alpha_{i,1}e_4+\frac{1}{2}\alpha_{i,7}e_{10}+ \frac{1}{2}\alpha_{i,8}e_{11}, \\

\varphi_{e_i}(e_5)=\alpha_{i,1}e_5-\frac{1}{2}\alpha_{i,6}e_{10}-\frac{1}{2}\alpha_{i,7}e_{11}, \\

\varphi_{e_i}(e_6)=\alpha_{i,1}e_6+\frac{1}{2}\alpha_{i,6}e_{11}, \\

\varphi_{e_i}(e_j)=\alpha_{i,1}e_j, \ \ 7\leq j\leq 11. \\
\end{cases}$$

Considering the property $\varphi_{e_i}(e_j)=e_i\cdot e_j=e_j\cdot e_i=\varphi_{e_j}(e_i)$, we obtain the following restrictions:

 $$\begin{array}{lll}
    \{e_1, e_2\}, & \Rightarrow & \alpha_{1,1}=\alpha_{2,1}=\alpha_{2,6}=0, \\[1mm]
     & &\alpha_{2,7}=-\alpha_{1,6}, \ \alpha_{2,8}=-\alpha_{1,7}, \ \alpha_{2,9}=\beta_{1,9}, \ \alpha_{2,10}=\beta_{1,10}, \ \alpha_{2,11}=\beta_{1,11},\\[1mm]
    \{e_1, e_3\}, & \Rightarrow & \alpha_{3,1}=\alpha_{3,6}=\alpha_{3,7}=\alpha_{3,8}=\alpha_{3,9}=0, \ \alpha_{3,10}=\frac{1}{2}\beta_{1,9}, \ \alpha_{3,11}=-\frac{1}{2}\alpha_{1,9},  \\[1mm]
    \{e_1, e_4\}, & \Rightarrow & \alpha_{4,1}=\alpha_{4,6}=\alpha_{4,7}=\alpha_{4,8}=\alpha_{4,9}=0, \ \alpha_{4,10}=\frac{1}{2}\alpha_{1,7}, \ \alpha_{4,11}=\frac{1}{2}\alpha_{1,8},  \\[1mm]
    \{e_1, e_5\}, & \Rightarrow & \alpha_{5,1}=\alpha_{5,6}=\alpha_{5,7}=\alpha_{5,8}=\alpha_{5,9}=0, \ \alpha_{5,10}=-\frac{1}{2}\alpha_{1,6}, \ \alpha_{5,11}=-\frac{1}{2}\alpha_{1,7},  \\[1mm]
   \{e_1, e_6\}, & \Rightarrow & \alpha_{6,1}=\alpha_{6,6}=\alpha_{6,7}= \alpha_{6,8}=\alpha_{6,9}=\alpha_{6,10}=0, \ \alpha_{6,11}=\frac{1}{2}\alpha_{1,6},  \\[1mm]
    \{e_1, e_j\}, & \Rightarrow & \alpha_{j,1}=\alpha_{j,6}=\alpha_{j,7}= \alpha_{j,8}=\alpha_{j,9}=\alpha_{j,10}=\alpha_{j,11}=0, \ 7\leq j\leq 9,  \\[1mm]

 \{e_2, e_3\}, & \Rightarrow & \beta_{3,9}=0, \ \beta_{3,10}=\frac{1}{2}\beta_{2,9}, \ \beta_{3,11}=-\frac{1}{2}\beta_{1,9},  \\[1mm]
 \{e_2, e_4\}, & \Rightarrow & \beta_{4,9}=0, \ \beta_{4,10}=-\frac{1}{2}\alpha_{1,6}, \ \beta_{4,11}=-\frac{1}{2}\alpha_{1,7},  \\[1mm]
 \{e_2, e_5\}, & \Rightarrow & \beta_{5,9}=\beta_{5,10}=0, \ \beta_{5,11}=\frac{1}{2}\alpha_{1,6},\\[1mm]
 \{e_2, e_j\}, & \Rightarrow & \beta_{j,9}=\beta_{j,10}=\beta_{j,11}=0, \ 6\leq j\leq 11.
    \end{array}$$
Thus, from the above restrictions we obtain the transposed Poisson algebra structure ${\bf TP}(g^3_{11})$ defined on Lie algebra  $g^3_{11}.$
\end{proof}

\end{document}